\documentclass[11pt]{amsart}

\usepackage{amssymb}
\usepackage{amsthm}
\usepackage{amsmath}
\usepackage{graphicx}
\usepackage{comment}
\usepackage[usenames,dvipsnames]{xcolor}
\usepackage[margin = 1in]{geometry}
\usepackage{enumerate}
\usepackage[toc,page]{appendix}
\usepackage{color}
\usepackage{hyperref}

\definecolor{mygreen}{rgb}{0.1,0.75,0.2}

 \newtheorem{thm}{Theorem}[section]
 \newtheorem{cor}[thm]{Corollary}
 \newtheorem{lem}[thm]{Lemma}
 \newtheorem{prop}[thm]{Proposition}

 \theoremstyle{definition}
 \newtheorem{defn}{Definition}

 \theoremstyle{remark}
 \newtheorem{rem}{Remark}

 \numberwithin{equation}{section}

\DeclareMathOperator{\ran}{Ran}
\DeclareMathOperator{\kerr}{Ker}

\newcommand{\la}{\langle}
\newcommand{\ra}{\rangle}
\newcommand{\pt}{\partial}
\newcommand{\eps}{\varepsilon}

\newcommand{\ud}{\,\mathrm{d}}
\newcommand{\8}{\infty}
\newcommand{\A}{\mathcal{A}}
\newcommand{\LL}{\mathcal{L}}
\newcommand{\E}{\mathcal{E}}
\newcommand{\F}{\mathcal{F}}
\newcommand{\Q}{\mathcal{Q}}
\newcommand{\ptf}{(-\pt_{xx})^{\frac{1}{2}}}

\DeclareMathOperator{\nul}{Null}

\begin{document}

\title[Long time behavior of PN model]{Long time behavior of dynamic solution to Peierls--Nabarro dislocation model}

\author{Yuan Gao}
\address{Department of Mathematics and Department of
  Physics\\Duke University,
  Durham NC 27708, USA}
\email{yg86@duke.edu}

\author{Jian-Guo Liu}
\address{Department of Mathematics and Department of
  Physics\\Duke University,
  Durham NC 27708, USA}
\email{jliu@math.duke.edu}

\date{\today}

\begin{abstract}
In this paper we study the relaxation process of  Peierls-Nabarro dislocation model, which is a gradient flow with singular nonlocal energy and double well potential  describing how  the materials relax to its equilibrium with the presence of a dislocation.   We prove the dynamic solution to Peierls-Nabarro model will converge exponentially to a shifted steady profile which is uniquely determined.

\end{abstract}

\maketitle

\section{Introduction}
\subsection*{Motivation and Problem. }
Materials defects such {as} dislocations   are important line defects in crystalline materials and they play essential roles in understanding materials properties like plastic deformation \cite{Mura, book}. For the single dislocation problem in materials science, Peierls-Nabarro (PN) model is used to describe the detailed structures in dislocation core, which is a multiscale continuum model that incorporates the {atomistic} effect by introducing a nonlinear nonconvex potential describing the  {atomistic} misfit interaction across the slip plane of the dislocation \cite{PN1, PN2}.

The simplest solvable nonlinear potential is introduced by Frenkel in 1926 to describe the misfit energy of the Halite \cite{Frenkel}.  Suppose $u$ is the displacement of materials.
Setting some physical constant to be $1$, under some symmetric assumption, the double well potential can be defined as
\begin{equation}
F(u):= \frac{1}{\pi^2}(1+\cos(\pi u))
\end{equation}
with
\begin{equation}
f(u):= F'(u)=-\frac{1}{\pi}\sin(\pi u),\, f'(u)=F''(u)=-\cos(\pi u).
\end{equation}
A more general version of $F$ is $F(u)=\frac{c}{\pi^2}(1+\cos(\frac{4\pi u}{b})).$
Due to the presence of dislocation on slip plane $\{x\in \mathbb{R}\}$, the total increment of displacement from $-\8$ to $+\8$ is $b$, the magnitude of the Burgers vector. Unlike the classical dislocation model \cite{Mura, book, Xiang}, which assumes a uniform increment of $u$ across slip plane, the increment of displacement $u$ at each position $x$ in PN model is not simply a step function but depends on the nonlinear misfit energy.
\begin{rem}
 In general  {a} Burgers vector, which indicates the magnitude and direction of the lattice distortion resulting from a dislocation, is defined by a loop integration $\textbf{b}:=\oint_L \ud u$ {with  a counterclockwise orientation enclosing the dislocation line}. If in two dimensions, we assume anti-symmetry with respect to  slip plane $\{x\in \mathbb{R}\}$, i.e. $u^+(x, 0^+)=-u^-(x,0^-)$. Due to Cauchy's integral formula, the loop integration is zero for  {the upper and lower half-spaces} separately.  Then by standard loop integration calculation, the loop integration is reduced to $x$-axis and is given by
 $2\int_{\alpha}^\beta u'(x) \ud x$, where $\alpha$ and $\beta$ are intersection points of the loop with $x$-axis. Therefore in PN model,  the distributional Burgers vector depends on the endpoints $\alpha,\,\beta$ we choose.
However, since the total increment from $-\8$ to $+\8$ remains to be $b$ and we always assume equilibrium at far field, the magnitude of the Burgers vector  naturally gives the boundary condition of $u$ at far field, $u(\pm\8)=\pm \frac{b}{4}$.
\end{rem}

 To find out displacement $u$ at each position,
the equilibrium of PN model for single edge dislocation is obtained by minimizing the total energy, including elastic bulk energy $E_{els}$ and misfit interface energy $\int_{\mathbb{R}} F(u) \ud x$.  By the Dirichlet to Neumann map and elastic extension \cite{our}, the elastic bulk energy in upper/lower plane $E_{els}$ can be reduced equivalently to slip plane, which therefore becomes a nonlocal elastic energy on slip plane $\{x\in \mathbb{R}\}$, $\int_{\mathbb{R}} \frac{1}{2}|(-\pt_{xx})^{\frac{1}{4}} u|^2 \ud x$; see \eqref{tm1.4} below.

Denote $H^s(\mathbb{R})$ as the fractional Sobolev space with norm denoted as $\|\cdot\|_s$.
Denote $\|\cdot\|$ as the standard $L^2(\mathbb{R})$ norm. We first give a singular integral definition, which is equivalent to Fourier's definition \cite{Ten}.
  For $0<s<1$, define the fractional Laplace operator $L_s$ from $D(L_s)=H^{2s}(\mathbb{R})\subset L^2(\mathbb{R})$ to $L^2(\mathbb{R})$
  \begin{equation}\label{def-L}
    L_s v := (-\pt_{xx})^s v := C_s{\mathrm{PV}} \int_{\mathbb{R}} \frac{v(x)-v(y)}{|x-y|^{1+2s}} \ud y,
  \end{equation}
  where $C_s$ is a normalizing constant to guarantee the symbol of the resulting operator is $|\xi|^{2s}.$ 
  Especially when $s=\frac{1}{2},$ $C_s=\frac{1}{\pi}.$ Although there are different equivalent definitions, we clarify we use the singular integral definition above in the whole paper.
Let us first express formally the problem we are interested in.  
  Define the nonlocal energy for Peierls-Nabarro model
\begin{equation}\label{tm1.4}
  E(u):= \int_{\mathbb{R}} \frac{1}{2}|(-\pt_{xx})^{\frac{1}{4}} u|^2 \ud x +  \int_{\mathbb{R}} F(u) \ud x.
\end{equation}
Alternatively, we can rewrite the nonlocal energy using a singular kernel
\begin{equation}\label{nonlocalE}
\begin{aligned}
  E(u)= &\frac{1}{4\pi} \int_{\mathbb{R}} \int_{\mathbb{R}}  \frac{(u(x)-u(y))^2}{|x-y|^2} \ud x \ud y +\int_{\mathbb{R}} F(u) \ud x\\
  =& \frac{1}{4} \int_{\mathbb{R}} \int_{\mathbb{R}}  J(x-y)(u(x)-u(y))^2 \ud x \ud y + \int_{\mathbb{R}} F(u) \ud x,
  \end{aligned}
\end{equation}
where $J(z)=\frac{1}{\pi z^2}$ and we used the identity
$$\frac{1}{2}\int_{\mathbb{R}} u\ptf  u  \ud x = \frac{1}{4\pi} \int_{\mathbb{R}} \int_{\mathbb{R}} \frac{(u(x)-u(y))^2}{|x-y|^2} \ud x \ud y.$$
Then the dynamic Peierls-Nabarro model is the following Allen-Cahn gradient flow
\begin{equation}\label{N}
  \pt_t u  = - \frac{\delta E(u)}{\delta u}= -\ptf  u - f(u)=-\A u,
\end{equation}
where the nonlocal nonlinear operator $\A$ formally defined as
\begin{equation}
  \A u : = \ptf  u + f(u)= L_{1/2} u+ f(u).
\end{equation}
Due to the presence of dislocation, with magnitude of Burgers vector $b=4$ in Remark 1, we are interested in solution with behavior at far field
\begin{equation}\label{bc1.8}
u(\pm\8,t)=\pm 1.
\end{equation}
The readers may see three main issues here. First, the displacement function $u$ is bounded but not vanish at far field. How does this boundary condition at far field remain as time evolving? Second, can the nonlocal operator $\ptf$ defined above on $H^1(\mathbb{R})$ be extended to  $L^\8(\mathbb{R})$ function with boundary condition \eqref{bc1.8}?
Third, non-vanishing boundary condition at far field leads to an infinite nonlocal elastic  energy $\int_{\mathbb{R}} \frac{1}{2}|(-\pt_{xx})^{1/4} u|^2 \ud x$ on slip plane (see footnote\footnote{There exists $A>0$ such that $u>\frac12$ for $x>A$ while $u<-\frac12$ for $x<-A$. Therefore $$\frac{1}{4\pi} \int_{\mathbb{R}} \int_{\mathbb{R}}   \frac{(u(x)-u(y))^2}{|x-y|^2} \ud x \ud y\geq \frac{1}{4\pi} \int_{x>A} \int_{y<-A}   \frac{1}{2(x^2+y^2)} \ud x \ud y=\8.$$} below), as well as an infinite  elastic bulk energy in upper and lower space, which is equivalently connected to the nonlocal elastic energy;  see precise statement in the perturbed sense established \cite{our}  by introducing a concept of elastic extension. This singularity in energy is analogous to the
vortex singularity in fluid mechanics or a single electron in electromagnetism, which inspires us to define a perturbed energy with respect to a reference state, steady profile described below.

We observe the typical bistable steady solution to \eqref{Neqn}, which will be used as reference state later.
Assume $\phi$ is the steady solution to \eqref{Neqn} satisfying
\begin{equation}\label{steay1.7}
  \A \phi =0,\quad \phi(\pm \8)=\pm 1.
\end{equation}
Since $\phi$ is smooth enough, we remark the operator $\ptf$ acting on $\phi$ is equivalent to $\ptf\phi=H(\phi')$ (see footnote\footnote{Since $\phi$ is uniformly bounded, only $y=x$ is the singular point in the singular integral definition \eqref{def-L}. Therefore $\ptf \phi:= \frac{1}{\pi} \mathrm{PV} \int_{\mathbb{R}} \frac{\phi(x)-\phi(y)}{|x-y|^2 } \ud y=\lim_{\eps\to 0} \frac{1}{\pi} \int_{|y-x|>\eps} \frac{\phi(x)-\phi(y)}{|x-y|^2 } \ud y =\lim_{\eps \to 0} \int _{|y-x|>\eps} \frac{\phi'(y)}{x-y}\ud y=H(\phi')$ due to integral by parts.}  below), where $H$ is the Hilbert transform
$$(Hu)(x):=\frac{1}{\pi} PV \int_{-\8}^{+\8} \frac{u(y)}{x-y}\ud y.$$
Indeed, $\phi(x)=\frac{2}{\pi}\arctan(x)$ is one special solution with fixed center at zero, i.e. $\phi(0)=0$.  Notice $f(\frac{2}{\pi}\arctan(x))=-\frac{1}{\pi}\sin (2 \arctan x)= -\frac{2}{\pi} \frac{x}{1+x^2}$ and $\ptf  \phi(x) = H (\phi'(x)) = H(\frac{2}{\pi}\frac{1}{1+x^2})=\frac{2}{\pi} \frac{x}{1+x^2}.$
 We can check
\begin{equation}\label{steadyPN}
\A \phi = \ptf  \phi + f(\phi)=0,
\end{equation}
and
\begin{equation}\label{phiasy}
\phi(x) \sim \pm 1 -\frac{2}{\pi x},\quad \text{ as } x\to \pm \8.
\end{equation}




In this paper,
we consider the long time behavior of solution to the dynamic equation
\eqref{N}
with initial data $u_0$ such that $u_0(\pm\8)=\pm 1$.
Our goal is to prove there is $x_0$ such that as $t\to \8$
$$u(x,t)\to \phi(x-x_0)$$
uniformly with exponential decay rate.
To make the infinity integrals meaningful, we define the perturbed  energy as
\begin{equation}\label{E-ref}
\E(u):= \frac{1}{2}\int_{\mathbb{R}} |(-\Delta)^{1/4} (u-\phi)|^2 -(u-\phi)f(\phi) + F(u) \ud x.
\end{equation}
We will study
\begin{equation}\label{Neqn}
  \pt_t u  = -\frac{\delta {\E}(u)}{\delta u}= -\A u
\end{equation}
 with initial data $u(x,0)=u_0(x)$ satisfying
\begin{eqnarray}
  &\text{ (i) } \E(u_0)<+\8; \label{in2}\\
  &\text{(ii) there exists constants } a\leq b \text{ such that }\nonumber\\
  &\phi(x-b)\leq u_0(x) \leq \phi(x-a)\label{in3}.
\end{eqnarray}
Thanks to the theory of analytic semigroup, we first validate this dynamic equation for $u$ by proving the global classical solution to the perturbation with respect to the reference profile, $v:=u-\phi$; see more details in Section \ref{sec2.1}.


\subsection*{Main Results and Related References.}
Below, we state the main result for uniform exponential convergence of dynamic solution to PN model to its equilibrium profile.
\begin{thm}\label{mainth}
  Assume initial data $u_0(x)-\phi(x)\in H^\frac{1}{2}(\mathbb{R})$ then \eqref{Neqn} has a unique global smooth solution $u(x,t)$. Furthermore, if  $u_0$ satisfies  \eqref{in2} and \eqref{in3}, then there exist constants $x_0$, $c$ and $\mu$ such that
\begin{equation}
  |u(x,t)-\phi(x-x_0)| \leq c \min\{\frac{1}{1+|x|}, e^{-\mu t}\}\quad  \text{ for any }t>0, \, x\in \mathbb{R}.
\end{equation}
\end{thm}

For stationary solutions to equilibrium PN model \eqref{steay1.7}, \cite{cabre2005} established the existence and uniqueness (upto a shift in $x$) of monotonic solutions by considering the corresponding local scalar problem by harmonic extension; see also \cite{cabre2015} for general nonlocal operator $(-\pt_{xx})^s,$ $0<s<1$. Recently, using different method \cite{valdinoci2013} also obtained the existence and uniqueness of monotonic solution and proved the monotonic solution is the global minimizer
of the nonlocal Allen-Cahn energy \eqref{nonlocalE} after  renormalization.
To connect the nonlocal Allen-Cahn equation \eqref{N} to the true vector field solution rigorously, rather than the analogous scalar model, \cite{our}  prove the equivalence between the nonlocal problem and the corresponding extended problem by defining a perturbed elastic bulk energy and establishing the elastic extension analogue to harmonic extension.

However, as far as we know the natural question proposed in the last section has not been studied, i.e. whether the dynamic solution to \eqref{Neqn} will converge uniformly to a  uniquely determined steady profile as $t\to +\8$.
The difficulties are essentially the singularity in energy, the lack of uniform in time $H^1$, as well as $L^2$, estimates and spectral gap analysis, which will be explained in details later.

Let us review here some related works among the vast literature of analysis for asymptotic behaviors. For the classical Allen-Cahn equation with double well potential, \cite{fife} proved the global exponential stability of  a traveling wave solution, which established the first framework to tackle the long time asymptotic behavior using  spectral gap analysis for diffusion operator linearized along traveling waves; see also \cite{pego} for invariant manifold method. Under the small perturbation assumption, \cite{Xin1, Xin2} proved the multidimensional stability of traveling wave  {solutions}. Furthermore, for nonlocal Allen-Cahn equation with nonsingular kernel, \cite{bates1} study the properties and travelling wave solutions as well as the uniform asympototic stability. For a class of integro-differential  {equations which contain} a nonlocal term expressed by {the} convolution of $u$ with some nonsingular kernel, \cite{chen1997} established an abstract theorem for uniqueness, existence and exponential stability of traveling wave {solutions} while \cite{chen2002} presents spectral analysis for linearized {operators} along traveling wave  {solutions} and obtain multidimensional stability for small perturbations. We are unaware of any asymptotic stability  results for nonlocal operator with singular kernel, whose steady profile has infinite energy.
As for the estimates for smallest eigenvalues of local or nonlocal Schr\"odinger operator,  we refer to \cite{frank, frank3, laskin, monneau2012,  frank-lieb, simon1990} and reference therein.  Let us also mention some results for  {dislocation models} in lager scale, described by dislocation density function. Analytic results such  {as} well-posedness for dislocation particle system, slow motion and  concentration of transition layers are established in \cite{DFV, monneau2012, DPV, PV0, PV1, PV2}.

\subsection*{Difficulties and Methods.}
The general idea is  to first prove the dynamic solution will uniformly converge to a shifted steady profile $\phi(x-x_0)$. Then by the spectral analysis for nonlocal  Schr\"odinger operator, which is linearized along the steady profile, we obtain the exponential decay rate.

The essential difficulties for the uniform convergence are compactness and characterization of limit set. As shown in the footnote in the previous page, we have an
 infinite nonlocal energy, which is only meaningful with the perturbed definition \eqref{E-ref}. However, we don't know if it has a lower bound. In other words, we do not have a uniform in time $L^2$ bound for the perturbed  solution although the semi-norm $\dot{H}^{1/2}$ is bounded. Moreover, unlike the local problem, we do not have a uniform in time $H^1(\mathbb{R})$ estimate, which is beyond the energy space. So we define a special $\omega$-limit set with vanishing dissipation; see Definition \ref{def-o}. For this kind of $\omega$-limit set,  which takes advantage of the vanishing dissipation property for a sequence of solution $u(x, t_n)$,  we have uniform estimate for $\|u(x,t_n)\|_{H^1}$ and can characterize the limit uniquely as a shifted steady profile $\phi(x-x_0)$; see Proposition \ref{stand}. Moreover, for the compactness of the solution, it is not the case for problems with local operator or nonlocal problems with nonsingular kernel so we can not obtain the compactness using modulus of continuity. By imposing the initial condition \eqref{in3} and thanks to the comparison principle and good decay properties for steady profile $\phi$, we obtain the compactness in Section \ref{sec2.3}. This, together with the characterization of $\omega$-limit set, leads to a convergence from $u(x, t_n)$ to $\phi(x-x_0)$. Notice the  vanishing dissipation property valids only for the  subsequence we extracted.  By further proving for any $t$ large enough, the solution will stay around the steady profile $\phi(x-x_0)$, we finally obtain the uniform convergence in Theorem \ref{Uni}.

Although the spectral analysis for the linearized  nonlocal Schr\"odinger operator is standard, we give a new proof  involving some particular global properties of the fractional Laplace operator, which allow us to construct a sequence of eigenfunctions with minimal points locating in the concave part of double well potential $F$; see Proposition \ref{prop-s3}. 
 The spectral gap obtained in Theorem \ref{gap} shows a lower bound for the norm of the linearized nonlocal operator for any $u$ orthogonal to $\phi'$. Using this property, we prove the exponential decay by first shifting the dynamic solution to the orthogonal space of some nonlocal  Schr\"odinger operator linearized along some steady profile $\phi(x-x_0-\alpha(t))$ in terms of a dynamic coordinate $x-\alpha(t)$ and then proving the shifting coordinate $\alpha(t)$ will converge to zero exponentially; see Section \ref{sec4} and Theorem \ref{mainth}. It worth to mention in the proof of Theorem \ref{mainth}, due to the lack of uniform in time $H^1$ bound, we play the same trick to first deal with the subsequence with vanishing dissipation.

\subsection*{Outlines} The rest of this paper is organized as follows. In Section \ref{sec2}, we will first prove  the uniform convergence of the dynamic solution $u(x,y)$ to its equilibrium, which is uniquely characterized as a shifted steady profile, i.e. $\phi(x-x_0)$. In Section \ref{sec3}, we establish the spectral decomposition for linearized nonlocal schr\"odinger operator, which leads to a spectral gap. All the proofs for the detailed spectral decomposition are  in Appendix \ref{App_B}. In Section \ref{sec4}, we combine the spectral gap with the uniform convergence to finally obtain the exponential decay of dynamical solution to its equilibrium $\phi(x-x_0)$.

\section{Uniform convergence from the dynamic solution to the steady profile $\phi$}\label{sec2}

This section will focus on the uniform convergence from the dynamic solution to its equilibrium, which involves essentially two main questions, compactness and characterization of the $\omega$-limit set. Here the $\omega$-limit set is a special one defined in Definition \ref{def-o}, which takes advantage of the property of solutions with a vanishing dissipation. For this kind of $\omega$-limit set, we can characterize it uniquely as a shifted steady profile $\phi(x-x_0)$ in Section \ref{sec2.2}. Then thanks to the compactness and stability guaranteed by comparison principle, we will obtain the uniform convergence to $\phi(x-x_0)$ in Section \ref{sec2.4}. We shall first clarify the existence and uniqueness of global classical solution to the dynamic problem \eqref{Neqn}.

\subsection{Global classical solution}\label{sec2.1}
Recall \eqref{Neqn} and $\A\phi=0$. Set perturbation function as
$$v(x,t):=u(x,t)-\phi(x).$$
 Then the dynamic equation for $v$ is
\begin{equation}\label{v-eq}
  \pt_t v  = -L_{\frac12} v - f(u) + f(\phi)
\end{equation}
with initial data $v_0(x)=u_0(x)-\phi(x),$ where $u_0(x)$ satisfies \eqref{in2} and \eqref{in3}.
Notice that if $u_0(x)$ satisfies \eqref{in2} and \eqref{in3}, then from $F(\cdot)\geq 0$ and $\|\phi(\cdot)\|<c$ we know
$v_0(x)\in H^{\frac{1}{2}}(\mathbb{R}).$
  We will use the theory for contraction semigroup to first establish the existence and uniqueness of global classical solution to \eqref{v-eq}.
Define the free energy for $v$ as
\begin{equation}\label{v-energy}
 \F(v):=\int \frac{1}{2}|(-\Delta)^{1/4} v|^2 -vf(\phi) + F(v+\phi) \ud x.
\end{equation}
Denote
  \begin{equation}\label{A2.3}
    A v := (L_{\frac{1}{2}}+I) v,
\quad
    G( v) := f(\phi)-f(v+\phi)+v.
  \end{equation}
Then \eqref{v-eq} becomes
  \begin{equation}\label{v-eq-n}
    \pt_t v  = \A \phi - \A u = -A v + G(v).
  \end{equation}

From now on, $c$  and $C$ will be  genetic constants whose values may change from line to line.
We have the following well-posedness result for \eqref{v-eq}. The proof is standard but to show the idea clearly, we give a brief proof in Appendix \ref{App_A} for $v_0\in H^1$. For the case $v_0\in H^{1/2}$, the idea is similar by analytic semigroup and we refer to \cite{Henry}.
 \begin{thm}\label{strongslu}
 Assume initial data $v_0(x):=u_0(x)-\phi(x)\in H^1(\mathbb{R}).$
 \begin{enumerate}[(i)]
    \item There exists global unique solution
 \begin{equation}
   v\in C^1([0,\8); L^2(\mathbb{R}))\cap C([0,\8); H^1(\mathbb{R})) \quad
 \end{equation}
to \eqref{v-eq-n} such that $v(x,0)=v_0(x)$ and $\pt_t v , Av, G(v)\in L^2(\mathbb{R})$ and the equation \eqref{v-eq-n} is satisfied in $L^2(\mathbb{R})$ for any $t>0$;
    \item the solution can be expressed by
\begin{equation}\label{mild}
  v(t) = e^{-At}v_0 + \int_0^t e^{-A(t-\tau)} G(v(\tau)) \ud \tau;
\end{equation}
    \item for any $k,j\in \mathbb{N}^+$ and $\delta >0$ there exist $C_{\delta, k, j},\, c$ such that
\begin{equation}\label{highreg}
\begin{aligned}
  v\in C^k((0,\8);H^j(\mathbb{R}));\\
  \|\pt_t^k v(\cdot, t) \|_j \leq C_{\delta, k,j} e^{c t}, \quad  t\geq \delta;
  \end{aligned}
\end{equation}
\item we have the energy identity
\begin{equation}\label{dissi}
 \frac{\ud \F(v(t))}{\ud t} =  - \int [-(-\Delta)^{1/2} v -f(v+\phi)+f(\phi)]^2 \ud x  =: -\mathcal{Q}(v(t))\leq 0.
\end{equation}
 \end{enumerate}
\end{thm}

%


\subsection{Characterization of $\omega$-limit set}\label{sec2.2}
In this section, we devote efforts to characterize the $\omega$-limit set whenever it is not empty. We will characterize it for sequence $u(x, t_n)$ with vanishing dissipation.
\begin{lem}[Vanishing sequence for dissipation]\label{DecayQ}
Assume $\F(t)$ is bounded from below and $\F'(t)\leq 0$. Let $\Q(t)=-\F'(t)$ defined in \eqref{dissi}.
Then there exists a subsequence $t_n \to +\8$ such that
\begin{equation}
  \Q(t_n)=-\F'(t_n) \to 0.
\end{equation}
\end{lem}
\begin{proof}
  Notice that the conclusion in the lemma is equivalent to
\begin{center}
    For any $\eps>0$, any $T>0$, there exists $t^*>T$ such that
$ -\eps < \F'(t) \leq 0 .$
\end{center}
Then we use the contradiction argument. If not, there exists $\eps_0>0$ and $T>0$ such that for any $t>T$, $\F'(t)<-\eps_0.$ It implies $\F(t)\to - \8$, which contradicts with $\F(t)$ is bounded from below.
\end{proof}

Now we define the special $\omega$-limit set as below.
\begin{defn}\label{def-o}
Assume $v(x,t)$ is the dynamic solution to \eqref{v-eq} with initial data $v_0\in H^\frac{1}{2}(\mathbb{R}).$ Let $\Q(t)=-\F'(t)$ defined in \eqref{dissi}. We define the $\omega$-limit set with vanishing dissipation as
\begin{equation}\label{omega}
  \omega(v):=\{v^*; \text{ there exist } t_n\to +\8 \text{ such that }
   v(\cdot, t_n)\to v^*(\cdot) \text{ in }L^2(\mathbb{R}) \text{ and }
   \Q(t_n)\to 0
   \},
\end{equation}
which is a subset of classical $\omega$-limit set.
\end{defn}

First we state a strict positivity property at global minima and global maxima for the nonlocal operator $\ptf$, which will be used later.
\begin{lem} (Strict positivity property at global minima and global maxima)\label{pp}
For any function $g(x)\in C(\mathbb{R})$, let $x_m, x_M \in (-\8,+\8)$ be the points where $g(x)$ attains it global minimum and maximum separately. Then we have
\begin{equation}
\ptf g(x)|_{x=x_m}<0, \quad \ptf g(x)|_{x=x_M}>0
\end{equation}
provided $g(x)$ is not a constant.
\end{lem}
\begin{proof}
From the definition of $\ptf$, since $g(x_m)\leq g(x)$ for all $x\in\mathbb{R}$,  we have
$$\ptf g(x)|_{x=x_m} \leq 0$$
and the equality holds only when $g(x)\equiv g(x_m)$ for all $x\in\mathbb{R}$. The proof for $\ptf g$ at $x_M$ is same.
\end{proof}

\begin{prop}[Characterization of $\omega$-limit set]\label{stand}
 Let $v$ be the dynamic solution to \eqref{v-eq} with initial data $v_0\in H^\frac{1}{2}(\mathbb{R}).$ Assume $\omega(v)\neq \emptyset$ and let $v^*\in \omega(v)$ defined in \eqref{omega}.  Then there exists $t_n\to +\8$ such that
  \begin{enumerate}[(i)]
    \item  $v(\cdot, t_n)\to v^*(\cdot) $ in $L^2(\mathbb{R})\cap L^\8(\mathbb{R})$;
    \item   $v^*\in H^1(\mathbb{R})$ is the steady solution to
     \begin{equation}\label{steady-eq}
     -\ptf  v^* = f(v^*+\phi)-f(\phi),
     \end{equation}
     in the sense that equation holds in $L^2(\mathbb{R})$;
    \item  $$\F(v^*)<+\8, \quad \lim_{x\to \pm\8} v^*(x)= 0;$$
    \item  moreover, there exists $x_0$ such that
    \begin{equation}
    v^*(x) = \phi(x-x_0)-\phi(x),\quad x\in\mathbb{R}.
    \end{equation}
  \end{enumerate}
\end{prop}
\begin{proof}
Step 1. Since $v^*\in\omega(v)$, we know there exist $t_n\to +\8$ such that $v(\cdot, t_n)\to v^*(\cdot)$ in $L^2(\mathbb{R})$. Thus  $\|v(t_n)\|\leq c$ and $\|v^*\|\leq c.$ Recall
\begin{equation}\label{tm4.32}
  \Q(t_n)=-\F'(t_n)=- \int [-(-\Delta)^{1/2} v -f(v+\phi)+f(\phi)]^2 \ud x \to 0.
\end{equation}
Therefore, $|\Q(t_n)|$ is bounded by $1$ for $n$ large enough  and thus
\begin{align*}
\|\ptf v(t_n)\|^2 \leq& -\|f(v(t_n)+\phi)-f(\phi)\|^2 + 2\int_{\mathbb{R}} |\ptf v(t_n) (f(v(t_n)+\phi)-f(\phi))| +1\\
 \leq &\|f(v(t_n)+\phi)-f(\phi)\|^2+ \frac{1}{2}\|\ptf v(t_n)\|^2 +1\\
  \leq &\max|f'|\|v(t_n)\|^2+ \frac{1}{2}\|\ptf v(t_n)\|^2 +1,
\end{align*}
which implies
\begin{equation}\label{tm-H1}
  \|v(\cdot, t_n)\|_{\dot{H}^1} \leq c.
\end{equation}
From Ladyzhenskaya's inequality, we have
\begin{equation}
\begin{aligned}
  \|v(\cdot, t_n)\|_{L^\8}&\leq \sqrt{2}\|v(\cdot, t_n)\|^{1/2}\|v(\cdot, t_n)\|^{1/2}_{\dot{H}^1}\\
&\leq c \|v(\cdot, t_n)\|^{1/2},
\end{aligned}
\end{equation}
which, after applying to $v(\cdot, t_n)-v^*(\cdot)$,
concludes (i).

Step 2.
 Notice  \eqref{tm-H1} and  $\|v(t_n)\|\leq c$. We have $\|v(t_n)\|_1^2$ is bounded and there exists a subsequence such that $v(\cdot, t_n)\rightharpoonup v^*(\cdot)$ in $H^1$ weakly.
Thus from the lower semi continuity of norm and $v(\cdot, t_n)\to v^*(\cdot) $ in $L^\8(\mathbb{R})$, we know
\begin{equation}
\int_{\mathbb{R}} [-\ptf  v^*(x) - f(v^*+\phi)+f(\phi)]^2 \ud x\leq \liminf_{t_n\to\8} \Q(t_n) \to 0,
\end{equation}
which concludes $v^*$ is the solution to \eqref{steady-eq}. Since also $f(\phi)\in L^2(\mathbb{R})$, \eqref{steady-eq} holds in $L^2$ sense and we concludes (ii). Recall free energy $\F(v)$ in \eqref{v-energy}. We obtain the bound for $\F(v^*)$ from  lower semi continuity of norm and $v(\cdot, t_n)\to v^*(\cdot) $ in $L^\8(\mathbb{R})$.  Moreover we know $v^*\in H^1(\mathbb{R})\hookrightarrow C^{0,\alpha}(\mathbb{R})$, $\alpha<\frac{1}{2}$ so $\lim_{x\to\pm\8}v^*(x)=0$ and  we conclude (iii).

Step 3. It remains to prove (iv) that all the steady solution $v^*(x)$ to \eqref{steady-eq} is exactly $\phi(x-x_0)-\phi(x)$ for some $x_0$.
Let $u^*(x):=v^*(x)+\phi(x)$. Since $\A\phi =0$ in classical sense and $v^*\in H^1(\mathbb{R})$, we know from \eqref{steady-eq} $u^*(x)$ is the solution to
\begin{equation}\label{steady-ueq}
  \ptf u^*(x) = f(u^*(x))
\end{equation}
in the sense that equation holds in $L^2(\mathbb{R})$. In two cases below, we will first prove $v^*(x)=\phi(x-x_0)-\phi(x)$ if $u^*\in(-1,1)$, then claim $u^*$ must be in $(-1,1)$ by contradiction argument.

Case 1. We assume $u^*(x)=v^*(x)+\phi(x)\in (-1,1)$. For any $\eps>0$, since $v^*(\pm\8)=0$ and $u^*(\pm\8)=\phi(\pm\8)=\pm 1$, there exist $x_\eps$ and $\xi_\eps$ such that
\begin{equation}
  v_\eps(x):=u^*(x)- \phi(x-x_\eps)+\eps\geq 0 \, \text{ for any }x\in\mathbb{R}
\end{equation}
and
\begin{equation}
  v_\eps(\xi_\eps)=u^*(\xi_\eps)- \phi(\xi_\eps-x_\eps)+\eps=0.
\end{equation}
If $v_\eps\equiv \text{const}$,
$$u(x)\equiv\phi(x-x_\eps)+\eps$$
for any $x\in\mathbb{R}$, which contradicts with $u(\pm\8)=\phi(\pm\8)=\pm 1$. Thus $v_\eps$ is not constant.

 Now we claim $x_\eps, \xi_\eps$ are both finite. Notice both $u^*(x)$ and $\phi(x-x_\eps)$ satisfy \eqref{steady-ueq}.  Since $v_\eps$ attains its minimum at $\xi_\eps$, by Lemma \ref{pp} we have
\begin{align}\label{eps}
  0> \ptf v_\eps(x)|_{x=\xi_\eps} &= \Big[-f(u^*(x)) + f(\phi(x-x_\eps))\Big] \Big|_{x=\xi_\eps}\\
  &= \Big[-f(\phi(x-x_\eps)-\eps) + f(\phi(x-x_\eps))\Big] \Big|_{x=\xi_\eps} =f'(\eta)\eps
\end{align}
with $$\eta\in [u^*(\xi_\eps),u^*(\xi_\eps)+\eps]=[\phi(\xi_\eps-x_\eps)-\eps , \phi(\xi_\eps-x_\eps)].$$
Therefore $\eta$ must locate in concave part of $F$, i.e. $\eta\in (-\frac{1}{2}, \frac{1}{2}).$
Then
\begin{equation}
u^*(\xi_\eps)\in (-\frac{1}{2}-\eps, \,\frac{1}{2})\subset[-\frac{3}{4}, \frac{1}{2}]
\end{equation}
for $\eps<\frac{1}{4}$.
Since $u^*(\cdot)\in (-1,1)$ is continuous function,
 so $\xi_\eps$ is bounded uniformly for $\eps<\frac{1}{4}$. On the other hand,
 we also have
 \begin{equation}
\phi(\xi_\eps-x_\eps) \in (-\frac{1}{2}, \,\frac{1}{2}+\eps)\subset[-\frac{1}{2}, \frac{3}{4}],
\end{equation}
  which implies $\xi_\eps-x_\eps\in [-2,2]$. This concludes $x_\eps, \xi_\eps$ are both bounded uniformly for $\eps<\frac{1}{4}$.

  Take $\eps\to 0$ and a convergent subsequence (still denote as $x_\eps, \xi_\eps$) such that   $x_\eps\to x_0$ and $\xi_\eps\to \xi$ for some $x_0$ and $\xi$. Clearly we still know $\xi-x_0\in[-2,2]$. Then we have
   \begin{equation}\label{ximini}
   \begin{aligned}
  &u^*(x)- \phi(x-x_0)\geq 0 \, \text{ for any }x\in\mathbb{R}\\
  &u^*(\xi)- \phi(\xi-x_0)=0.
  \end{aligned}
\end{equation}
   From \eqref{eps}, we know
  \begin{equation}
     0\geq \ptf (u^*(x)-\phi(x-x_0))|_{x=\xi}=\lim_{\eps\to 0}\ptf v_\eps|_{x=\xi_\eps}=\lim_{\eps\to 0}f'(\eta)\eps = 0
  \end{equation}
  This, together with $\xi$ attains the minimum  by \eqref{ximini}, leads to
  $$
  u^*(x)-\phi(x-x_0) \equiv const = 0\, \text{ for all } x\in \mathbb{R},
  $$
  which means $u^*(x)\equiv \phi(x-x_0)$ and $v^*(x)\equiv \phi(x-x_0)-\phi(x)$.

  Case 2. We assume  $u^*(x)=v^*(x)+\phi(x)\notin (-1,1)$ for some $x$. We use contradiction argument to see it is not possible. We only deal with the left side, i.e. $u^*(x)=v^*(x)+\phi(x)\leq -1$ for some $x$. The argument for the other side $u^*(x)=v^*(x)+\phi(x)\geq 1$ is same.

  Since $u^*$ is continuous function connecting from $-1$ to $-1$, then if $u^*\leq -1$, it can attain its minimal point at some  finite $x^*$.  Assume
  \begin{equation}
    u^*(x^*)=\min_{x\in\mathbb{R}} u^* \in (-1-2k, 1-2k]\quad \text{ for some }k\in \mathbb{N}^+.
  \end{equation}
  First, from the \eqref{steady-ueq}, $u^*(x^*)\neq 1-2k$. Otherwise by Lemma \ref{pp},
  \begin{equation}
    0 = \ptf u^*(x)|_{x=x^*}-f(1-2k)=\ptf u^*(x)|_{x=x^*}<0
  \end{equation}
  leads to a contradiction. Then we know $u^*(x^*)=\min_{x\in\mathbb{R}} u^*(x)\in (-1-2k, 1-2k),$  Therefore, we choose $\eta$  such that $ u^*(x)+2k\geq \phi(x-\eta)$ for any $x\in \mathbb{R}$ and $u^*(x)+2k$ touches $\phi(x-\eta)$ at the point $x_1$, i.e.
  \begin{equation}
    \left\{
       \begin{array}{ll}
         u^*(x)+2k\geq \phi(x-\eta) & \hbox{ for }x\in \mathbb{R};  \\
         u^*(x_1)+2k= \phi(x_1-\eta).
       \end{array}
     \right.
  \end{equation}
  Notice the minimal point $x^*$ is finite so $x_1, \eta$ is finite. Since $f$ is $2k$-periodic function, we have
  \begin{align*}
    0=& \Big[\ptf (u^*(x)+2k)-f(u^*(x)+2k) ~-~ \ptf (\phi(x-\eta))+f(\phi(x-\eta)) \Big]\Big|_{x=x_1}\\
    =& \ptf \Big(u^*(x)+2k -  \phi(x-\eta) \Big)\Big|_{x=x_1}<0,
  \end{align*}
  where we used Lemma \ref{pp} again.
This also gives a contradiction and we complete the proof of (iv).
\end{proof}


%
%
%

\subsection{Comparison Principle and Compactness}\label{sec2.3}

In the previous section, we have seen clearly the characterization of $\omega$-limit set with vanishing dissipation whenever it is not empty. However, in order to extract such a sequence $v(t_n)$ with a limit in $\omega(v)$ defined in \eqref{omega}, we need compactness in $L^2$. One possible way to  achieve it is  the comparison principle.

\subsubsection{Comparison Principle }

We have the following comparison principle.
\begin{prop}
Let initial data satisfies assumption \eqref{in3}. Then \begin{equation}\label{squ-s}
  \phi(x-b)  \le u(x,t) \le \phi(x-a) , \quad \forall x \in  \mathbb R, t > 0,
\end{equation}
where $b\geq a$ are constants given in \eqref{in3}.
\end{prop}
\begin{proof}
We only prove the left hand side of \eqref{squ-s}.
Denote $w(x,t):=u(x,t)-\phi(x-b)$. Then we know
\begin{align*}
  &\pt_t w =-\ptf w + f(\phi(\cdot-b))-f(w+\phi(\cdot-b)),\\
  &w(\cdot, 0)\geq 0.
\end{align*}
Assume $t^*$ is the first time such that $w$ attain zero at some point $x^*$. Therefore
\begin{align*}
  &w(x,t)\geq 0 \quad \text{ for any }0\leq t\leq t^*, \, x\in \mathbb{R};\\
  &w(x^*, t^*)=0.
\end{align*}
Then at $x=x^*$, $t=t^*$
\begin{align*}
  \pt_t w |_{(x^*,t^*)}&= -\ptf w|_{(x^*,t^*)} + f(\phi(x^*-b))- f(\phi(x^*-b)+w(x^*,t^*))\\
  &= -\ptf w|_{(x^*,t^*)}
  \ge 0,
\end{align*}
where we used $w(x^*,t^*)$ is the minimum. Moreover, since $w(\pm\8)=0$, $w$ can not be a nontrivial constant. Therefore by Lemma \ref{pp} $\pt_t w |_{x^*,t^*}>0$ and we conclude $w(x,t)=u(x,t)-\phi(x-b)\geq 0$  all the time.
\end{proof}

\begin{lem}[Basic decay estimate at far field]\label{highes}
There exists a positive constant $C$ such that for any dynamic solution $u(x,t)$ to \eqref{Neqn} with initial data satisfying  \eqref{in2} and \eqref{in3},
  \begin{equation}
    |1-u(x,t)|,\,|f(u)| < \frac{C}{1+|x|}, \quad x>0, \, t>0;
  \end{equation}
    \begin{equation}
    |1+u(x,t)|,\,|f(u)| < \frac{C}{1+|x|}, \quad x<0, \, t>0.
  \end{equation}
\end{lem}
\begin{proof}
From \eqref{squ-s}, we obtain the basic estimate for $u$,
\begin{align*}
  |1-u(x,t)|\leq |1-\phi(x-b)|\leq \frac{C}{1+|x|},\,\, x>0,
\end{align*}
where we use the asymptotic estimate \eqref{phiasy}. Similarly we have,
\begin{align*}
  |-1-u(x,t)|\leq \frac{C}{1+|x|},\,\, x<0.
\end{align*}

Moreover, we obtain the basic estimate for nonlinear term
\begin{align*}
|f(u)|&=\frac{1}{\pi}| \sin(\pi u)| \\
&= \frac{1}{\pi}|\sin(\pi(1-u))|= \frac{1}{\pi}|\sin(\pi(1+u))|\\
&\leq \left\{
        \begin{array}{ll}
           \frac{C}{1+|x|} , & \hbox{ for } x>0; \\
          \frac{C}{1+|x|} , & \hbox{ for } x<0.
        \end{array}
      \right.
\end{align*}
\end{proof}

\subsubsection{Compactness}
Now we turn to prove the compactness in $L^2(\mathbb{R})$,which is  the  key point to guarantee $\omega$-limit set is not empty.
\begin{lem}[Compactness]\label{lem-com}
  Assume  $u(x,t)$ is the dynamic solution to \eqref{Neqn} with initial data satisfying  \eqref{in2} and \eqref{in3}. For each $\delta>0$ the set of functions
  $$\{u(\cdot, t)-\phi(\cdot); \, t\geq \delta\}$$
  is relatively compact in $L^2(\mathbb{R}).$
\end{lem}
\begin{proof}
  Step 1. For any $\eps>0$, from Lemma \ref{highes}, we can  choose $K$ such that for $|x|>K,\, t>0$
 \begin{align*}
& \|u-\phi\|_{L^2(|x|>K)}\\
 \leq & \|u-\phi\|_{L^2(x>K)} + \|u-\phi\|_{L^2(x<-K)}\\
 \leq & \|u-1\|_{L^2(x>K)}+\|1-\phi\|_{L^2(x>K)}+\|u+1\|_{L^2(x<-K)}+\|-1-\phi\|_{L^2(x<-K)}\\
 \leq & c \int_{|x|>K} \frac{1}{(1+|x|)^2} \ud x < \frac{\eps}{2}.
 \end{align*}

  Step 2.   Recall free energy for $v=u-\phi$
  \begin{equation}
 \F(v)= \int \frac{1}{2}|(-\Delta)^{1/4} v|^2 -vf(\phi) + F(v+\phi) \ud x
\end{equation}
   and energy identity \eqref{dissi}.
   Since $F(v+\phi)\geq 0$ and $\F(v(t))\leq \F(v_0)$, we know
   \begin{align}\label{H12}
     \int_{\mathbb{R}} \frac{1}{2}|(-\Delta)^{1/4} v|^2\ud x \leq c+ \|v\|\|f(\phi)\|\leq c,
   \end{align}
   where we also used $\|v\|\leq c$ by Lemma \ref{highes}.
   Thus the compact embedding $H^{\frac12}(-K,K)\hookrightarrow\hookrightarrow L^2(-K,K)$ shows there exists a subsequence $t_n\to +\8$ such that $u(\cdot, t_n)-\phi(\cdot) \to u^*(\cdot)-\phi(\cdot)$ in $L^2(-K, K).$
Therefore, $\lim_{n\to \8} u(x,t_n)-\phi(x)=u^*(x)-\phi$ in $L^2(\mathbb{R}).$
\end{proof}

{
\begin{rem}
It worth to notice the initial condition \eqref{in3} is only used to obtain the uniform in time estimate for $u$ at far field.
As we have seen in the proof of Lemma \ref{lem-com}, the compactness result can be achieved as long as we have the uniform in time $L^2$ bound.
It is another possible way to relax the initial condition \eqref{in3}.
\end{rem}
}

\subsection{Stability and Uniform Convergence}\label{sec2.4}
We have obtained the compactness in $L^2$ and the characterization of $\omega$-limit set in previous preparations. Therefore, (i) we can first extract a  sequence $u(x,t_n)-\phi(x)$ with vanishing dissipation $\Q(t_n)$ by Lemma \ref{DecayQ}, (ii) then by compactness Lemma \ref{lem-com} $u(x,t_n)-\phi(x)$ possesses further a subsequence such that the limit of $u(x, t_{n_k})-\phi(x)$ is in $\omega(v)$, in other words, for any $v_0\in H^{\frac12}(\mathbb{R})$, $\omega(v)\neq \emptyset$. However those properties are only for some subsequence $t_n$. In this section, we are finally in the position to obtain the uniform convergence by proving the dynamic solution will stay close to the standing profile for all large time.
  First we list some properties for the double well function $F(x).$

Since $f'(\pm 1)>0$,
there exist $\mu>0$,
$\delta>0$ such that for $0<q<\frac{\delta}{2}$,
\begin{equation}\label{con_mu}
f(\phi)-f(\phi-q)\geq \mu q\quad \text{ for  }1-\delta \leq \phi \leq 1 \text{ or }-1 \leq \phi \leq -1+\delta.
\end{equation}
Moreover, for $\phi\in[-1+\delta, 1-\delta]$, there exist $k>0$, $\beta\geq 0$ such that
\begin{equation}\label{con_k}
 |f(\phi-q)-f(\phi)|\leq kq \quad \text{ for any } 0<q<\frac{\delta}{2},
\end{equation}
 and
\begin{equation}\label{con_beta}
  \phi'(x) \geq \beta>0 \quad \text{for } x \text{ such that } \phi(x)\in[-1+\delta, 1-\delta].
\end{equation}

\begin{prop}[Stability]\label{lem4.7}
  Assume $u(x,t)$ is a dynamic solution to \eqref{Neqn} and for any $0<\eps< \frac{\delta}{2}$ there exists $N$ such that
  $$\sup_{x\in \mathbb{R}}|u(\cdot, t_N)-\phi(\cdot-x_0)|<\eps.$$
  Then for any $t>t_N$,
  there exists $C$ such that
   $$\sup_{x\in\mathbb{R}}|u(x, t)-\phi(x-x_0)|<C\eps.$$
   Moreover
\begin{equation}\label{squ}
  \phi(x-x_0 - \frac{\mu+k}{\mu \beta}\eps) - \eps e^{-\mu (t-t_N)} \le u(x,t)
 \le \phi(x-x_0 +\frac{\mu+k}{\mu \beta}\eps) + \eps e^{-\mu (t-t_N)}, \quad \forall x \in  \mathbb R, t > t_N.
\end{equation}
\end{prop}
\begin{proof}
 We will use comparison principle to prove for $t>t_N$ the solution still stay close to $\phi(x-x_0).$
First we prove the lower bound for $u$. Notice
  \begin{equation}\label{q0}
    \phi(x-x_0)-\eps\leq u(x,t_N)\quad \text{ for any }x\in \mathbb{R}.
  \end{equation}
  We  construct a subsolution
  \begin{equation}
    \underline{u}(x,t):=\max\{-1, \phi(x-\xi(t))-q(t)\}\in [-1,1]
  \end{equation}
  by choosing $\xi(t)$ and $q(t)$ such that $q(t):=\eps e^{-\mu (t-t_N)}$, $\xi(t):=c_1+c_2e^{-\mu (t-t_N)}$ with $c_1=x_0-c_2$ and $c_2<0$ to be determined.

Define
\begin{equation}
  N(u):= \pt_t u  + \A u = \pt_t u  +\ptf u + f(u)
\end{equation}
and divide $[-1,1]$ into several sets
$$
  I_1:=\{(x,t); ~\phi(x-\xi(t)) \in [-1, -1+q(t)] \},
  $$
  $$
  I_2:=\{(x,t); ~\phi(x-\xi(t)) \in [-1+q(t), -1+\delta] \},
  $$
  $$
  I_3:=\{(x,t); ~\phi(x-\xi(t)) \in [-1+\delta, 1-\delta] \},
  $$
  $$
  I_4:=\{(x,t); ~\phi(x-\xi(t)) \in [1-\delta, 1] \}.
  $$

(1) If $(x,t)\in I_1$, then $\phi(x-\xi)-q(t)\leq -1$ and $N(\underline{u})=0.$

(2) If $(x,t)\in I_2,$
since $\A\phi=0$, $\xi'\geq 0$ and \eqref{con_mu}, we know
\begin{align*}
  N(\underline{u})&= -\phi'(x-\xi(t)) \xi' -q' +\ptf \phi(x-\xi(t)) +f(\phi(x-\xi(t))-q)\\
&=-\phi'(x-\xi(t))\xi' - q' - f(\phi(x-\xi(t))) + f(\phi(x-\xi(t))-q)\\
&\leq -\phi'(x-\xi(t)) \xi' -q' -\mu q \,\,\\
&\leq -q'-\mu q=0.
\end{align*}
The situation for $(x,t)\in I_4$ is exactly same.

(3) For $(x,t)\in I_3$, i.e. $-1+\delta \leq \phi(x-\xi(t)) \leq 1-\delta$, from \eqref{con_beta} and \eqref{con_k} we know
\begin{align*}
  N(\underline{u})&\leq -\phi'(x-\xi(t)) \xi' -q' + kq\\
&\leq -\beta \xi' - q' + kq.
\end{align*}
Set $\xi' = \frac{-q'+kq}{\beta}=\frac{\mu+k}{\beta}q>0,$ we have
$$c_1=x_0-c_2,\quad c_2=-\frac{\mu+k}{\mu\beta} \eps.$$

Then $N(\underline{u})\leq 0$ and $\underline{u}$ is a subsolution satisfying
\begin{equation}
  u(x,t)\geq \underline{u}(x,t) \geq \phi(x-\xi(t))-q(t)
\end{equation}
due to comparison principle.
Therefore we have
\begin{equation}
  \phi(x-x_0 - \frac{\mu+k}{\mu \beta}\eps) - \eps e^{-\mu (t-t_N)} \le u(x,t), \quad \forall x \in  \mathbb R, t > t_N.
\end{equation}

Similarly, we can obtain the upper bound for $u$
\begin{equation}
 u(x,t) \le \phi(x-x_0 +\frac{\mu+k}{\mu \beta}\eps) + \eps e^{-\mu (t-t_N)}, \quad \forall x \in  \mathbb R, t > t_N.
\end{equation}
Hence we know
\begin{align*}
    |u(x,t)-\phi(x-x_0)| \le \max_{x\in\mathbb{R}}\phi'(x)\cdot \frac{\mu+k}{\mu \beta}\eps + \eps , \quad \forall x \in  \mathbb R, t > t_N,
\end{align*}
which concludes for $C=1+\frac{2}{\pi}\frac{\mu+k}{\mu \beta}$ we have
   $$\sup_{x\in \mathbb{R}}|u(x, t)-\phi(x-x_0)|<C\eps$$
   for any $t>t_N$.
\end{proof}

After all the preparations above, we can first extract a time sequence $t_n$ with vanishing dissipation $\Q(t_n)$ by Lemma \ref{DecayQ} and then by compactness Lemma \ref{lem-com} we can further extract a subsequence such that the limit of $u(x, t_n)-\phi(x)$ is in $\omega(v)$. Moreover, $u(x, t)-\phi(x)$ will stay close to its limit for any $t$ large enough.

\begin{thm}[Uniform Convergence]\label{Uni}
Assume  $u(x,t)$ is the dynamic solution to \eqref{Neqn} with initial data satisfying  \eqref{in2} and \eqref{in3}.
  Then there exists a value $x_0$ such that
  \begin{equation}\label{uni-con}
  \lim_{t\to + \8} b(t) = 0, \quad b(t):=\max_{x\in \mathbb{R}} |u(x,t)-\phi(x-x_0)|.
  \end{equation}
\end{thm}
\begin{proof}
Recall $v(x,t)=u(x,t)-\phi(x)$ with the free energy
\begin{equation*}
 \F(v)= \frac{1}{2}\int |(-\Delta)^{1/4} v|^2 -vf(\phi) + F(v+\phi) \ud x.
\end{equation*}
Then by Lemma \ref{highes} we know $\|v\|\leq c$ and thus $\F(v)$ is bounded from below. Therefore, combining energy identity \eqref{dissi} and Lemma \ref{DecayQ} leads to a vanishing sequence for $\Q$, i.e.
there exists a time sequence $t_n \to +\8$ such that
\begin{equation}
  \Q(t_n)=-\F'(t_n) \to 0.
\end{equation}
For such a sequence $t_n$, from Lemma \ref{lem-com} we know
$$\{u(\cdot, t_n)-\phi(\cdot), \quad t_n\geq \delta\}$$
is relative compact in $L^2(\mathbb{R}).$ Therefore we know the $\omega$-limit set $\omega(v)\neq \emptyset$ and the limit of the subsequence (still denote as $t_n$) $v(x,t_n)=u(x,t_n)-\phi(x)\to v^*$ can be characterized by Proposition \ref{stand} (iv), i.e. $v^*(x)=\phi(x-x_0)-\phi(x)$ and thus
$$u(x,t_n)-\phi(x-x_0)=v(x, t_n)+\phi(x) -\phi(x-x_0) \to 0 \quad \text{ in }L^\8(\mathbb{R}).$$

 Next, from the stability Proposition \ref{lem4.7}, we conclude the uniform convergence \eqref{uni-con}.
\end{proof}

\section{Spectral decomposition for linearized nonlocal Schr\"odinger operator }\label{sec3}

In this section, we will study detailed structures for spectrum of linearized nonlocal Schr\"odinger operator and prove the spectral gap in Proposition \ref{gap}.
Note $f'(\phi)= -\cos(\pi \phi) = \frac{x^2 -1}{x^2+1}.$ The linearized operator along the steady profile $\phi$
is $L: D(L)\subset L^2 \to L^2$ with
\begin{equation} \label{L1}
 L u := \ptf  u + \frac{x^2 -1}{x^2+1} u.
\end{equation}

Denote $\sigma_p, \, \sigma_r$ and $\sigma_c$ as the point spectrum, the residual spectrum and the continuous spectrum separately.
Then
$$\mathbb{C}=\rho(L)\cup\sigma(L)=\rho(L)\cup\sigma_p(L)\cup\sigma_c(L)\cup\sigma_r(L).$$
We will first prove there is no residual spectrum and all the continuous spectrum locate in $[1, +\8)$, see Proposition \ref{prop-s1} and Proposition \ref{prop-s2} separately. Although the proof is standard but for completeness we put them in Appendix \ref{app-s1} and Appendix \ref{app-s2}.
\begin{prop}\label{prop-s1}
For linear operator $L$ in \eqref{L1}, the spectrum $\sigma(L)=\sigma_p(L)\cup\sigma_c(L) \subset [-1, +\8)$.
\end{prop}

\begin{prop}\label{prop-s2}
For linear operator $L$ in \eqref{L1}, the continuous spectrum $\sigma_c(L)\subset [1, +\8)$ .
\end{prop}

Next proposition is the key procedure to prove $0$ is the principle eigenvalue and there is no other kinds of spectra  near zero. The proof is standard contradiction argument but it takes advantage of  strict positivity property at global minima and global maxima for nonlocal operator (see Lemma \ref{pp}), which allow us to construct a sequence of eigenfunctions with minimal points locating in the concave part of double well potential $F$.
\begin{prop}\label{prop-s3}
For linear operator $L$ in \eqref{L1}, the point spectrum $\sigma_p(L)\subset [0, +\8)$ and $0$ is simple eigenvalue with eigenfunction $\phi'(x)$.
\end{prop}
\begin{proof}
  Step 1.
 We prove $0$ is simple eigenvalue with eigenfunction $\phi'(x)$.
First, by differentiating $\A \phi = 0$ once, it is straightforward that
  $ \phi'(x)=\frac{2}{\pi}\frac{1}{1+x^2}$ is an eigenfunction corresponding to the eigenvalue $0$.

 Assume there is another eigenfunction $g$ corresponding to $0$ such that $g\in L^2$.
  By the regularity of the steady  solution, we know for any $k>0$, $g\in H^k(\mathbb{R})$ thus $g$ is smooth function.
  Without loss of generality, we assume $g$ takes  positive values at some $x_0$ (otherwise we can always construct such a function with some positive points by linear combination).
  Below, we will show $g$ is linearly dependent on $\phi'$.

  Define
  \begin{equation}
  \phi_\beta := \phi' + \beta g, \quad \beta\in \mathbb{R}.
  \end{equation}
  Define the set
  $$D_1:= \{\beta<0; \phi_\beta(\xi) <0 \text{ for some $\xi$} \}.$$
  Let
  $$
  \bar{\beta}:= \sup D_1.
  $$
  Such a $\bar{\beta}$ is well-defined. Indeed, since $g$ is positive at $x_0$, we know $\bar{\beta}\in [\beta_1, 0]$ with $\beta_1= -\frac{\phi'(x_0)}{g(x_0)}<0$.

Notice that if $\phi_\beta$ is a constant, since $\phi_\beta\in L^2$, we know $\phi_\beta\equiv 0$, which concludes $\phi'$ and $g$ are linearly dependent. Therefore, we can simply assume $\phi_\beta$ is not a constant.

  For any $\beta\in D_1$, since $\phi_\beta$ is also eigenfunction corresponding to eigenvalue $0$,
  \begin{equation}\label{eigen_1}
    L \phi_\beta =\ptf  \phi_\beta + f'(\phi) \phi_\beta= 0.
  \end{equation}
  Let $\xi_\beta\in[-\8, +\8]$ be a point such that $\phi_\beta$ attains its minimum. Thus we know $\phi_\beta(\xi_{\beta})<0.$
   Consider two cases (i) $\xi_\beta\in (-\8, +\8)$; (ii) $\xi_\beta=-\8$ or $+\8$. For case (ii), since $\phi_\beta\in L^2(\mathbb{R})$ and $\phi_\beta\in H^1(\mathbb{R})\hookrightarrow C(\mathbb{R})$, $\phi_\beta(\pm\8)$ must be zero, which contradicts with $\phi_\beta(\xi_{\beta})<0.$

   For case (i), by Lemma \ref{pp} we have
  $$
  \ptf  \phi_\beta |_{x=\xi_\beta}=\frac{1}{\pi}PV \int \frac{\phi_\beta(x)-\phi_\beta(y)}{|x-y|^2} \ud y \Big|_{x=\xi_\beta} { < }0.
  $$
  From \eqref{eigen_1} we know
  \begin{equation}
    f'(\phi)\phi_\beta \big|_{x=\xi_\beta} >0.
  \end{equation}
  which, together with $\phi_\beta(\xi_{\beta})<0$, leads to
  $$f'(\phi)|_{x=\xi_\beta}< 0.$$
   Due to the concave part of $F$ is bounded  between $-\frac12$ and $\frac12$, we know the set of $\xi_\beta$ is bounded. Indeed, $f'(\phi)(x)=\frac{x^2-1}{x^2+1}<0$ if and only if $x\in (-1,1)$.

  Take a convergent subsequence (still denote as $\beta$) with limit $\beta\to \bar{\beta}$ and $\xi_{\beta}\to \bar{\xi}$ for some $\bar{\xi}\in[-1,1]$. From the definition of $\bar{\beta}$, we know
  \begin{equation}
    \phi_{\bar{\beta}}(\bar{\xi}) = 0 \leq \phi_{\bar{\beta}}(\xi) \quad \text{ for any }\xi\in \mathbb{R}.
  \end{equation}
  Therefore from
  $L \phi_{\bar{\beta}}=0$
  we have
  $$\ptf  \phi_{\bar{\beta}} |_{x=\bar{\xi}}= - f(\phi) \phi_{\bar{\beta}} |_{x=\bar{\xi}}=0.$$
  However by Lemma \ref{pp}
  \begin{equation}
    \ptf  \phi_{\bar{\beta}} |_{x=\bar{\xi}} = \frac{1}{\pi}PV \int \frac{\phi_{\bar{\beta}}(x)-\phi_{\bar{\beta}}(y)}{|x-y|^2} \ud y \Big|_{x=\bar{\xi}} \leq 0.
  \end{equation}
  Therefore
  $$
  \phi_{\bar{\beta}}\equiv const = 0,
  $$
  which means $\phi'$ and $g$ are linearly dependent.

  Step 2. We prove $0$ is the principle eigenvalue.
  Assume $\lambda<0$ is the eigenvalue such that
  $$L u = \lambda u $$
  for some $u\in L^2$ and $u\neq 0$.
  By the regularity of the steady  solution, we know for any $k>0$, $u\in H^k(\mathbb{R})$ thus $u$ is smooth function and $|u|$ is continuous function.
  Then
  \begin{equation}\label{tm3.21}
    \ptf |u|+f(\phi)|u|\leq \mbox{sgn}{u}\cdot\big[\ptf u+f(\phi)u \big]=\lambda |u| \leq 0.
  \end{equation}
  Similarly,
  define
  \begin{equation}
  \phi_\beta := \phi' + \beta |u|, \quad \beta\in \mathbb{R}
  \end{equation}
and the set
  $$D_1:= \{\beta<0; \phi_\beta(\xi) <0 \text{ for some $\xi$} \}.$$
  Let
  $$
  \bar{\beta}= \sup D_1.
  $$
  which is well-defined since $|u|$ is positive at $x_0$ and  we know $\bar{\beta}\in [\beta_1, 0]$ with $\beta_1= -\frac{\phi'(x_0)}{|u|(x_0)}<0$.

Notice if $\phi_\beta$ is a constant, since $\phi_\beta\in L^2$, we know $\phi_\beta\equiv 0$, which concludes $\phi'$ and $|u|$ are linearly dependent, i.e. $L|u|=0$. However from \eqref{tm3.21}, $L|u|\leq \lambda |u|\leq 0$ and thus $\lambda=0$. It contradicts with $\lambda<0.$ Therefore we can simply assume $\phi_\beta$ is not a constant.

  For $\beta\in D_1$, from \eqref{tm3.21}
  \begin{align}\label{geq0}
    L \phi_\beta = \beta L |u|=\beta [\ptf |u|+f(\phi)|u|]\geq \beta\lambda|u|\geq 0, \text{ for all }x\in \mathbb{R}.
  \end{align}
  Let $\xi_\beta\in[-\8, +\8]$ be a point such that $\phi_\beta$ attains its minimum. Thus $\phi_\beta(\xi_{\beta})<0.$
   Consider two cases (i) $\xi_\beta\in (-\8, +\8)$; (ii) $\xi_\beta=-\8$ or $+\8$.
   For case (ii), since $\phi_\beta\in L^2(\mathbb{R})$ and $\phi_\beta\in H^1(\mathbb{R})\hookrightarrow C(\mathbb{R})$, $\phi_\beta(\pm\8)$ must be zero, which contradicts with $\phi_\beta(\xi_{\beta})<0.$

For case (ii),  by Lemma \ref{pp} we have
  $$
  \ptf  \phi_\beta |_{x=\xi_\beta}=\frac{1}{\pi}PV \int \frac{\phi_\beta(x)-\phi_\beta(y)}{|x-y|^2} \ud y \Big|_{x=\xi_\beta} { < }0.
  $$
  This, together with \eqref{geq0}, we know
  \begin{equation}
    f'(\phi)\phi_\beta \big|_{x=\xi_\beta} >0.
  \end{equation}
  Notice also $\phi_\beta(\xi_{\beta})<0$, thus
  $$f'(\phi)|_{x=\xi_\beta}< 0.$$
   Due to the concave part of $F$ is bounded between $-\frac12$ and $\frac12$, we know the set of $\xi_\beta$ is bounded, especially, $f'(\phi)(x)=\frac{x^2-1}{x^2+1}<0$ if and only if $x\in (-1,1)$.

 Take a convergent subsequence (still denote as $\beta$) with  limit $\beta\to \bar{\beta}$ and $\xi_{\beta}\to \xi^*$ for some $\xi^*\in[-1,1]$. From the definition of $\bar{\beta}$,
    $$\phi_{\bar{\beta}}(\xi^*)=0, \quad \phi_{\bar{\beta}}(x)\geq 0 \text{ for any }x\in \mathbb{R}.$$
    Then the limit of \eqref{geq0} shows that
    $$0\leq L \phi_{\bar{\beta}}= \ptf \phi_{\bar{\beta}} +f(\phi)\phi_{\bar{\beta}}.$$
    However at $x=\xi^*$, the RHS is
    $$
    \ptf \phi_{\bar{\beta}}|_{x=\xi^*} +f(\phi)\phi_{\bar{\beta}}|_{x=x^*}=\ptf \phi_{\bar{\beta}}|_{x=\xi^*}\leq  0.
    $$
    Therefore $\ptf \phi_{\bar{\beta}}|_{x=\xi^*}=  0$ and thus $$\phi_{\bar{\beta}}\equiv \phi_{\bar{\beta}}(\xi^*)=0,$$
    which means $\lambda$ could only be zero and contradicts with $\lambda<0.$
\end{proof}

%
%

From the Proposition \ref{prop-s2} and \ref{prop-s3} above, we know $0$ is the principle, simple eigenvalue of $L$ and the continuous spectrum $\sigma_c(L)\subset[1,+\8)$. Thus
we obtain spectral gap for nonlocal Schr\"oinger operator below.
\begin{thm}[Spectral gap]\label{gap}
  For linear operator $L$ in \eqref{L1}, there exists a constant $\lambda_2>0$ such that for any $u\bot \nul (L)$, i.e. $\int_\mathbb{R} u(x)\phi'(x) \ud x=0$, we have
  \begin{equation}
    \la L u, u \ra \geq \lambda_2\|u\|^2.
  \end{equation}
\end{thm}

\begin{rem}[Hardy type functional inequality and best constant]
Recall Hardy's inequality for the homogeneous Sobolev space in one dimension. For $0<s<\frac{1}{2}$
\begin{equation}
  \|u\|_{\dot{H}^s}^2 \geq C_{s} \int_{\mathbb{R}} |x|^{-2s} |u(x)|^2 \ud x,
\end{equation}
with sharp constant
$$C_s=2^{2s} \frac{(\Gamma(1+2s)/4)^2}{(\Gamma(1-2s)/4)^2}.$$
As a consequence of Proposition \ref{prop-s3}, we have the following Hardy's type functional inequality at critical index $s=\frac{1}{2}$.
\begin{cor}
  For any $u\in H^{1/2}(\mathbb{R}),$ we have
  \begin{equation}
  \int_{\mathbb{R}}\frac{1-x^2}{1+x^2} u^2(x) \ud x \leq \|u\|_{\dot{H}^{\frac{1}{2}}}^2.
  \end{equation}
  Moreover, the equality holds if and only if $u(x)=\frac{C}{1+x^2}.$
\end{cor}
\end{rem}

\begin{rem}
Notice by harmonic extension the steady profile in upper half plane is $\phi(x,y):=\frac{2}{\pi}\arctan\frac{x}{1+y}$, which has the harmonic conjugate $g(x,y):=\frac{1}{\pi}\ln(x^2+(1+y)^2).$ So $z(x,t):=\phi(x,y)+ig(x,y)$ is the holomorphic extension in the upper half-space $\mathbb{C}_+$ of $\phi(x)=\frac{2}{\pi}\arctan x$.
For the linearized problem,  a related holomorphic eigenvalue problem in  $\mathbb{C}_+$ is
  \begin{equation}
  -i  \pt_z w - \frac{2i}{i+z} w =  \lambda w,
  \end{equation}
  whose restriction on the real line becomes a nonlocal eigenvalue problem
  \begin{equation}
    \ptf  u + \frac{x^2-1}{x^2+1}u+ \frac{2x}{x^2+1}H(u)=(\lambda+1)u.
  \end{equation}
\end{rem}

\section{Exponential decay to steady profile}\label{sec4}
Next we will use the spectral gap Theorem \ref{gap} to prove the exponential decay rate for $u(x,t).$
To take advantage the lower bound of the linearized nonlocal operator $L$ for functions orthogonal to its null space $\nul(L)$, we need to first shift the standing profile in terms of a dynamic coordinate.  We   construct a shift function $\alpha(t)$ such that
\begin{equation}\label{def-shift}
  \phi_\alpha(x,t) := \phi(x-x_0- \alpha(t)),  \quad v_\alpha(x,t) := u(x,t) - \phi_\alpha(x,t)
\end{equation}
satisfy
$$
v_\alpha \bot \mbox{ Null} (\LL_\alpha ), \quad \LL_\alpha  := \ptf + f'(\phi_\alpha)
$$
i.e.
\begin{equation}\label{botcon}
\int_{-\8}^{\8} \big( u(x,t)-\phi(x-x_0-\alpha(t)) \big)  \phi'(x-x_0-\alpha(t)) \,dx
= \int_{-\8}^{\8}v_\alpha(x,t) \phi'(x-x_0-\alpha(t)) \,dx
=0.
\end{equation}
Notice that $\int_{-\8}^{\8} \phi(x) \phi'(x) \ud x =0$.
Define a functional of $\alpha$ as
\begin{equation}
W(t,\alpha) :
=\int_{-\8}^{\8}  u(x,t) \phi'(x-x_0-\alpha) \,\ud x.
\end{equation}
The following proposition is to clarify the existence, uniqueness and properties of  $\alpha(t)$ and it also provides an elementary proof for implicit function theorem in unbounded domain.
\begin{prop}\label{th4.1}
For $W(t, \alpha)$ in \eqref{botcon}, there exist $T>0$ large enough  and a unique $\alpha(t)$ such that
\begin{enumerate}[(i)]
  \item $W(t, \alpha(t))=0 \quad \text{ for }t>T$;
  \item $\alpha(t) \to 0$ \,\,as\, $t\to+\8$;
  \item $\alpha(t)\in C^1(T,+\8)$.
\end{enumerate}
\end{prop}
\begin{proof}
Step 1. We prove the existence and bound of $\alpha(t)$. Using intermediate value theorem and  \eqref{squ}, we will first prove there exists $T>0$ such that for any $t>T$ there exist at least one $\alpha(t)$ such that $W(t, \alpha(t))=0$ \text{ for } $t>T$. Moreover, for all the solutions to $W(t,\alpha(t))=0$, there exist $a_T,\, b_T$ such that $\alpha(t)\in[a_T, b_T]$ for $t>T$.

By  \eqref{squ}, for $t>T$ large enough and $\eps$ small enough, we know that there exists $x_0$ such that
$$
 \int_{-\8}^{\8} \big( \phi(x-x_0-\frac{\mu+k}{\mu \beta}\eps) - \eps e^{-\mu t}  \big)  \phi'(x-x_0-\alpha) \,dx
 \le W(t,\alpha),
$$
$$
  W(t,\alpha) \le \int_{-\8}^{\8} \big( \phi(x-x_0+\frac{\mu+k}{\mu \beta}\eps) + \eps e^{-\mu t}  \big)  \phi'(x-x_0-\alpha) \,dx,
$$
or equivalently
\begin{equation}\label{tm-sq1}
 \int_{-\8}^{\8} \big( \phi(x +\alpha-\frac{\mu+k}{\mu \beta}\eps) - \eps e^{-\mu t}  \big)  \phi'(x) \,dx
 \le W(t,\alpha),
\end{equation}
\begin{equation}\label{tm-sq2}
  W(t,\alpha) \le \int_{-\8}^{\8} \big( \phi(x +\alpha+\frac{\mu+k}{\mu \beta}\eps) + \eps e^{-\mu t}  \big)  \phi'(x) \,dx.
\end{equation}
We choose $T$ such that $\eps e^{-\mu T}<\frac12.$
Therefore
$$
 1< \int_{-\8}^{\8} \big( 1 - \eps e^{-\mu t}  \big)  \phi'(x) \,dx
 \le \lim_{\alpha \to\8}W(t,\alpha)
$$
$$
   \lim_{\alpha \to -\8}W(t,\alpha) \le \int_{-\8}^{\8} \big( -1 + \eps e^{-\mu t}\big)  \phi'(x) \,dx < -1,
$$
for any $t>T.$
Hence by intermediate value theorem there is at least one $\alpha(t)$.

Next, define $b_T$ is the solution of
$$\int_{-\8}^{\8}  \phi(x +b_T-\frac{\mu+k}{\mu \beta}\eps)\phi'(x) \ud x = 2\eps e^{-\mu T}<1$$
and $a_T$ is the solution of
$$\int_{-\8}^{\8}  \phi(x +a_T+\frac{\mu+k}{\mu \beta}\eps)\phi'(x) \ud x = -2\eps e^{-\mu T}>-1.$$
From \eqref{tm-sq1},
\begin{align*}
0&=W(t,\alpha(t))\geq \int_{-\8}^{\8}  \phi\left(x +\alpha(t)-\frac{\mu+k}{\mu \beta}\eps\right)\phi'(x) \ud x-2\eps e^{-\mu t}\\
\geq & \int_{-\8}^{\8}  \phi\left(x+\alpha(t)-\frac{\mu+k}{\mu \beta}\eps\right)\phi'(x) \ud x-2\eps e^{-\mu T}\\
\geq & \int_{-\8}^{\8}  \phi\left(x+\alpha(t)-\frac{\mu+k}{\mu \beta}\eps\right)\phi'(x) \ud x-\int_{-\8}^{\8}  \phi\left(x +b_T-\frac{\mu+k}{\mu \beta}\eps\right)\phi'(x) \ud x.
 \end{align*}
This implies $\alpha(t)\leq b_T$ since $\int_{-\8}^{\8}  \phi(x+\alpha-\frac{\mu+k}{\mu \beta}\eps)\phi'(x) \ud x$ is increasing with respect to $\alpha$.  Similarly, we can use \eqref{tm-sq2} to obtain $\alpha(t)\geq a_T$ so $a_T\leq \alpha(t)\leq b_T.$

Step 2. Uniqueness of $\alpha(t)$.
Differentiating $G$ with respect $\alpha$ yields
\begin{align*}
\pt_\alpha W&=\int_{-\8}^{\8} -u(x,t) \phi''(x-x_0-\alpha)\ud x\\
&=\int_{-\8}^{\8} \phi'(x+\alpha)  \phi'(x) \,\ud x
- \int_{-\8}^{\8} \big(u(x,t) - \phi(x-x_0)\big)  \phi''(x-x_0-\alpha) \,\ud x\\
&\ge \int_{-\8}^{\8} \phi'(x+\alpha)  \phi'(x) \,\ud x
-  \max_{x} |u(x,t) - \phi(x-x_0)| \int_{-\8}^{\8} |\phi''(x)| \,\ud x > 0
\end{align*}
for large $t>T_2$. Here we used $b(t)=\max_{x\in\mathbb{R}}|u(x,t)-\phi(x-x_0)|\to 0$ as $t\to+\8$ from Theorem \ref{Uni}.

Step 3. We prove $\alpha(t)\to 0$ as $t\to +\8.$

If $\alpha(t) \not \to 0$ as $t\to+\8$, then there are constant $a>0$ and a sequence $t_k\to\8$ as $k\to \8$ such that
$b_T\geq \alpha(t_k) \ge a$ (or $a_T\le \alpha(t_k) \le -a$). Then we have a subsequence (still denote as $t_k$) and $a^*>0$ such that $\alpha(t_k)\to a^*.$ Recall \eqref{botcon}, which shows
$$\int_{-\8}^{\8} \big(  \phi(x-x_0) -  \phi(x-x_0-\alpha(t_k)  \big)  \phi'(x-x_0-\alpha(t_k)) \,dx
=\int_{-\8}^{\8} \big(\phi(x-x_0) -  u(x,t_k)\big)  \phi'(x-x_0-\alpha(t_k)) \,dx.$$
Taking limit as $t_k\to\8$ in
\begin{align*}
&\int_{-\8}^{\8} \big(  \phi(x-x_0) -  \phi(x-x_0-\alpha(t_k)  \big)  \phi'(x-x_0-\alpha(t_k)) \,dx\\
=&\int_{-\8}^{\8} \big(\phi(x-x_0) -  u(x,t_k)\big)  \phi'(x-x_0-\alpha(t_k)) \,dx\\
\le& \max_{x} |\phi(x-x_0) -  u(x,t_k)| \int_{-\8}^{\8}  \phi'(x-x_0-\alpha(t_k)) \,dx\\
=& 2 \max_{x} |\phi(x-x_0) -  u(x,t_k)| \to 0,
\end{align*}
leads to
\begin{equation}\label{tm11}
\int_{-\8}^{\8} \big(  \phi(x-x_0) -  \phi(x-x_0-a^*)  \big)  \phi'(x-x_0-a^*)\ud x\leq 0.
\end{equation}
On the other hand, since $a^*>0$
\begin{equation}\label{tm12}
\phi(x-x_0) -  \phi(x-x_0-a^*)\geq 0.
\end{equation}
Then due to $\phi'> 0$, \eqref{tm11} and \eqref{tm12} lead to
$$\int_{-\8}^{\8} \big(  \phi(x-x_0) -  \phi(x-x_0-a^*)  \big)  \phi'(x-x_0-a^*) \ud x= 0,$$
which is a contradiction due to $a^*>0$.

Step 4.
$\alpha(t) \in C^1(T,+\8)$
is directly from implicity function theorem.
\end{proof}
Next, we prove the shift $\alpha(t)$ introduced above contributes an exponentially small error.
\begin{lem}\label{lem-exp}
For $\alpha(t)$ and $v_\alpha(x,t)$ defined in \eqref{def-shift},
there are constants $C$ and $\mu$ such that
\begin{itemize}
  \item $\|v_\alpha\| \le C e^{-\mu t};$
  \item $|\alpha(t)| \le C e^{-\mu t}$.
\end{itemize}
\end{lem}
\begin{proof}
Step 1. Decay of $\|v_\alpha\|$.
From Theorem \ref{Uni}, we have $b(t)=\max_{x\in\mathbb{R}}|u(x,t)-\phi(x-x_0)|\to 0$. Since
\begin{equation}
 \max_x |v_\alpha(x,t)| \le b(t) + c_1 \alpha(t)
\end{equation}
for $c_1:=\max_{x\in\mathbb{R}}\phi'(x)=\frac{2}{\pi},$ we have
\begin{equation}\label{vdecay}
\max_x|v_\alpha(x,t)|\to 0, \quad\text{as }t\to +\8
\end{equation}
due to Proposition \ref{th4.1} (ii).
  From the definition of $v_\alpha$, for any $x\in\mathbb{R}$,
\begin{eqnarray}
  \pt_t v_\alpha &=& \pt_t u + \dot{\alpha} \pt_x \phi_\alpha \nonumber \\
   &=& -\A u  + \A \phi_\alpha + \dot{\alpha}\pt_x \phi_\alpha \nonumber\\
   &=& -\LL_\alpha   u - f(u) + f'(\phi_\alpha)u  + \A \phi_\alpha+ \dot{\alpha} \pt_x \phi_\alpha \nonumber\\
   &=& -\LL_\alpha  v_\alpha - f(\phi_\alpha + v_\alpha)  + f(\phi_\alpha)+ f'(\phi_\alpha)v_\alpha + \dot{\alpha} \pt_x \phi_\alpha \nonumber\\
   &=& -\LL_\alpha  v_\alpha - \tfrac12 f''(\xi) v_\alpha^2 + \dot{\alpha} \pt_x \phi_\alpha,  \label{tm1.31}
\end{eqnarray}
where $\xi:=\xi(x)$ locates between $\phi_\alpha(x)$ and $\phi_\alpha(x)+ v_\alpha(x,t)$.

Since the shift $\alpha(t)\to 0$ and the upper/lower bound of $f'(\phi_\alpha)=-\cos(\pi \phi_\alpha)$ remains same,   we can directly apply spectral gap Theorem \ref{gap} to $\LL_\alpha $ to obtain
$$\la \LL_\alpha v_\alpha, v_\alpha \ra\geq \lambda_2 \|v_\alpha\|^2.$$
Therefore, multiplying $v_\alpha$ to both sides of \eqref{tm1.31} and integrating with respect to $x$ lead to
\begin{eqnarray}\label{expii}
  \frac{\ud}{\ud t}\frac{1}{2}\|v_\alpha(\cdot,t)\|^2 \leq -\lambda_2 \|v_\alpha(\cdot,t)\|^2 - \tfrac12 \int_{\mathbb{R}} f''(\xi(x))v_\alpha^3(x,t) \ud x,
\end{eqnarray}
where we used  $\la \phi_\alpha', v_\alpha \ra=0$.
For the second term $\int f''(\xi(x))v_\alpha(x,t)^3 \ud x$, from \eqref{vdecay},
$$|\int_\mathbb{R} f''(\xi(x))v^3_\alpha(x,t) \ud x|\leq \max_{x\in \mathbb{R}} (v_\alpha(x,t) f''(\xi(x))) \|v_\alpha(\cdot,t)\|^2 \leq \frac{\lambda_2}{2} \|v_\alpha(\cdot,t)\|^2 $$
for $t$ large enough. Therefore, \eqref{expii} gives the exponential decay rate for $v_\alpha$
 \begin{equation}
   \|v_\alpha(\cdot,t)\|\leq C e^{-\mu t}.
 \end{equation}

(ii) Decay of $\alpha(t)$.
Multiply $\eqref{tm1.31}$ by $\pt_x\phi_\alpha$, then we have
\begin{eqnarray}\label{alpha0}
  \la \pt_x\phi_\alpha, \pt_t v_\alpha \ra &=& \la -\LL_\alpha  v_\alpha+ \dot{\alpha} \pt_x\phi_\alpha - \tfrac12 f''(\xi) v_\alpha^2,~ \pt_x\phi_\alpha \ra \\
   &=& \dot{\alpha} \|\pt_x\phi_\alpha\|^2 - \int_{\mathbb{R}}  \tfrac12 f''(\xi) v_\alpha^2 \pt_x\phi_\alpha \ud x\nonumber,
\end{eqnarray}
where we used $\LL_\alpha  \pt_x\phi_\alpha=\LL_\alpha \phi_\alpha'=0$.  Differentiating the relation
\eqref{botcon} with respect to $t$ leads to
\begin{equation}
  \int_{\mathbb{R}} \phi_\alpha' \pt_t v_\alpha \ud x = \dot{\alpha} \int_{\mathbb{R}} \phi_\alpha'' v_\alpha \ud x,
\end{equation}
which is the LHS of \eqref{alpha0}. Thus \eqref{alpha0} becomes
$$\dot{\alpha}\|\phi_\alpha'\|^2 = \dot{\alpha} \int_{\mathbb{R}} \phi_\alpha'' v_\alpha\ud x + \int_{\mathbb{R}} \tfrac12 f''(\xi) v_\alpha^2 \phi_\alpha' \ud x.$$
This, together with the decay of $\|v_\alpha\|$ in Step 1, shows
\begin{align*}
  |\dot{\alpha}|\|\phi_\alpha'\|^2
  \leq  |\dot{\alpha}|\|\phi_\alpha''\|\|v_\alpha\|+ C\max|\phi_\alpha'| \|v_\alpha\|^2
  \leq  C e^{-\mu t}.
\end{align*}
Notice also $\|\phi_\alpha'\|^2= \|\phi'\|^2 = \int \frac{4}{2\pi^2}\frac{1}{(1+x^2)^2}\ud x$ and $\alpha(+\8)=0$.
Then standard calculus gives the exponential decay of $|\alpha(t)|$.
\end{proof}

Finally we collect all the results above and complete the proof of Theorem \ref{mainth}.
\begin{proof}[Proof of Theorem \ref{mainth}]
In Lemma \ref{lem-exp}, we proved
$$
 \|v_\alpha\|\leq C e^{-\mu t}.  
$$
Since $\phi\in H^1(\mathbb{R})$, then from Theorem \ref{strongslu} we know $u\in C(0,\8;H^1(\mathbb{R}))$ and thus $v_\alpha\in C(0,\8;H^1(\mathbb{R}))$. However, the uniform $H^1(\mathbb{R})$ bound for $u$, as well as $v_\alpha$, is only valid for some sequence $t_n$. Therefore we need to use the same trick in the proof of Theorem \ref{Uni} as explained below. From Theorem \ref{Uni}, there exists a  time sequence $t_n\to+\8$ such that $\Q(t_n)\to 0$ and thus  $\|u(\cdot, t_n)\|_{\dot{H}^1}\leq c$ due to \eqref{tm-H1}.   Applying Ladyzhenskaya's inequality to $v_\alpha(\cdot, t_n)$, we have
\begin{equation}
\begin{aligned}
  \|v_\alpha(\cdot, t_n)\|_{L^\8}&\leq \sqrt{2}\|v_\alpha(\cdot, t_n)\|^{1/2}\|v_\alpha(\cdot, t_n)\|^{1/2}_{\dot{H}^1}\\
&\leq c \|v_\alpha(\cdot, t_n)\|^{1/2}.
\end{aligned}
\end{equation}
Hence we obtain the pointwise decay rate for the subsequence $v_\alpha(x, t_n)$
\begin{equation}\label{se-8}
  \|v_\alpha(\cdot,t_n)\|_{L^{\8}}\leq  C e^{-\mu t_n}
\end{equation}
for $t_n$ large enough.
Notice
$
\phi'(x)= \frac{2}{\pi} \frac{1}{1+x^2} , $
which has a maximum $\frac{2}{\pi}.$
Thus
\begin{align*}
  |u(x,t_n)-\phi(x,x_0)|
  \leq & |v_\alpha(x, t_n)|+|\phi(x-x_0)-\phi(x-x_0-\alpha(t_n))|\\
  \leq & |v_\alpha(x, t_n)| + \max_{x\in \mathbb{R}} |\phi'||\alpha(t_n)|
  \leq  |v_\alpha(x, t_n)| + \frac{2}{\pi}  |\alpha(t_n)| \leq C e^{-\mu t_n},
\end{align*}
uniformly in $x$ due to Lemma \ref{lem-exp} and \eqref{se-8}.
Then by the stability result Lemma \ref{lem4.7}, we know for any $t$ large enough,
\begin{equation}
|u(x,t)-\phi(x,x_0)|< c  e^{-\mu t} \quad \text{uniformly in }x\in \mathbb{R}.
\end{equation}

Notice also the basic estimate for $u$ in Lemma \ref{highes} which gives
$$\sup_{x\in \mathbb{R}} |u(x,t)-\phi(x-x_0)| \leq \frac{c}{1+|x|}, \quad \text{ for any }t>0.$$
We complete the proof of the main Theorem \ref{mainth}.
\end{proof}

\begin{rem}
{ {To the end, we discuss the relation to the classical Benjamin-Ono equation.} Benjamin-Ono equation is  a nonlinear partial integrodifferential equation describing one dimensional internal waves in deep water. Consider
\begin{equation}\label{BOc}
  h_t=\ptf h_x+h_x - 2hh_x.
\end{equation}
Denote $$E_B(h):=\frac{1}{2}\la \ptf h, h\ra +\frac{h^2}{2}-\frac{h^3}{3}$$ with
$$\frac{\delta E_B}{\delta h}=-\ptf h+h-h^2.$$
Then \eqref{BOc} becomes
$$h_t= \pt_x(\frac{\delta E_B}{\delta h}),$$
which is a Hamiltonian system. If we consider a special one-parameter family transformation $T_c$ such that
$$h(x,t)=T_cu:=u(x-ct),$$
then we have the traveling wave form of the Benjamin-Ono equation
$$\ptf u=u^2+(-1-c)u.$$ 
Let $W:=(1+c)\frac{u^2}{2}-\frac{u^3}{3}$ with $W'=(1+c)u-u^2$ and $W''=1+c-2u$. The special traveling wave form of Benjamin-Ono equation is
\begin{equation}\label{BO}
 \ptf u = u^2 -(1+c)u = -W'(u),
\end{equation}
which is closed to our static equation \eqref{steadyPN}. 
Benjamin \cite{Ben} found that $\Phi_c=\frac{2(c+1)}{1+(1+c)^2x^2}$ is a solitary solution to \eqref{BO} which, apart from periodic solution, is unique up to translation  \cite{Toland1}. For instance for $c=0$, notice $\Phi^2-\Phi=\frac{2(1-x^2)}{(1+x^2)^2}$ and $(H(\Phi))'=(\frac{2x}{1+x^2})'$, then $\Phi=\frac{2}{1+x^2}$ is a solitary solution.

Define the linearized operator of Benjamin-Ono equation along its solitary profile $\Phi_c$ as
\begin{equation}
  L_B u:=\ptf u + W''(\Phi)u = \ptf u + \frac{(1+c)^3x^2-3(c+1)}{(1+c)^2x^2+1}u,
\end{equation}
whose potential $\frac{(1+c)^3x^2-3(c+1)}{(1+c)^2x^2+1}$ is very similar to our problem $\frac{x^2-1}{x^2+1}$ in Section \ref{sec3}; with lower bound $-3(c+1)$ and upper bound $1+c$. The spectral analysis for this kind of self-adjoint operator like $L_B$ and $L$ defined in \eqref{L1} is standard.  But for completeness, we give a new proof  involving some particular global properties of the fractional Laplace operator; see Proposition \ref{prop-s3}.

One may also notice that unlike the solitary profile to Benjamin-Ono equation which vanishes at far field, in PN model the steady profile to \eqref{steadyPN} is a transition connecting from $-1$ to $1$ due to the double well potential.  The dynamic PN model is a gradient flow while the Benjamin-Ono equation is a Hamiltonian flow. However, the steady profile are closely related, the derivative of $\pi\phi$ is exactly $\Phi(x)$; see more connection in \cite{TolandPN}. We refer to \cite{Ben5} for the orbit stability of solitary solution  to  \eqref{BOc}; see also \cite{Wein1987, Soug1987, Sha1987} for more general integrodifferential equation.
}
\end{rem}

\section*{Acknowledge}
The authors wish to thank Professor R. Pego for valuable discussions and tremendous insights.
J.-G. Liu was supported in part by the National Science Foundation (NSF) under award DMS-1812573 and the NSF grant RNMS-1107444 (KI-Net).

%
%

\appendix
\section{Proof of Theorem \ref{strongslu}}\label{App_A}
\begin{proof}[Proof of Theorem \ref{strongslu}]
  {Step 1. We collect some properties for $G$ defined in \eqref{A2.3}.}
\\
(a) $G: L^2(\mathbb{R})\to L^2(\mathbb{R})$  is global Lipschiz, i.e.  \begin{equation}\label{GLip}
\|G(v_1)-G(v_2)\|\leq (1+\max|f'|)\|v_1-v_2\|\leq 2 \|v_1-v_2\|.
\end{equation}
(b) If $v(\cdot)\in H^1(\mathbb{R})$, then $G(v(\cdot))\in H^1(\mathbb{R})$. Indeed,
\begin{equation*}
 \|\pt_x G(v)\|\leq \|\pt_x v \|+ \pi \|v\|,
\end{equation*}
which implies
\begin{equation}
  \|G(v)\|_1 \leq (\pi+2)\|v\|_1.
\end{equation}

Step 2.   First it is easy to check operator $A$ defined in \eqref{A2.3} is m-accretive in $L^2(\mathbb{R})$. Indeed we know $\text{Re}\la Ax, x\ra\geq 0$ for all $x\in D(A)$ and from Lemma \ref{lem-sc}, we know $\sigma(A)=[1,+\8).$ Therefore $A$ is an infinitesimal generator of a linear strongly continuous semigroup of contractions and $\|e^{-At}\|\leq 1$.
  Second from global Lipschitz condition \eqref{GLip},  there exists a unique mild solution
  expressed by  \eqref{mild} and $v\in C([0,+\8); L^2(\mathbb{R}))$.

  Step 3.
 Lipschitz continuity in $t$ of $v$ and $G(v)$.
  \begin{align}\label{tma3}
    &v(t+h)-v(t)\\
    =&e^{-At}(e^{-Ah} v_0 - v_0)+ \int_0^{t+h} e^{-A(t+h-\tau)}G(v(\tau))\ud \tau- \int_0^t e^{-A(t-\tau)} G(v(\tau)) \ud \tau\nonumber\\
    =&e^{-At}\big[(e^{-Ah} v_0 - v_0)+\int_0^h e^{-A(h-\tau)}G(v(\tau)) \ud  \tau\big]+\int_0^t e^{-A(t-\tau)} [G(v(\tau+h))-G(v(\tau))] \ud \tau\nonumber\\
    =&e^{-At}(v(h)-v_0)+ \int_0^t e^{-A(t-\tau)} [G(v(\tau+h))-G(v(\tau))] \ud \tau \nonumber
  \end{align}
  Since $\|e^{-At}\|\leq 1$,
    \begin{align*}
    \|v(t+h)-v(t)\|
    \leq \|v(h)-v_0\| +  \int_0^t 2\|v(\tau+h)-v(\tau)\|\ud \tau.
  \end{align*}
  Then by Gronwall's inequality, we have
  \begin{equation}\label{vLip_0}
    \|v(t+h)-v(t)\|\leq \|v(h)-v_0\|e^{2t}.
  \end{equation}
On the other hand,
\begin{align}\label{At2}
 {v(h)-v_0} = (e^{-Ah}-I)v_0 +  \int_0^h e^{-A(h-\tau)} [G(v(\tau))-G(v_0)+G(v_0) ]\ud \tau.
\end{align}
Then from \eqref{GLip} and $\|e^{-At}\|\leq 1$ we know
\begin{equation}
  \|v(h)-v_0\|\leq h\|Av_0\|+ 2\int_0^h \|v(\tau)-v_0\| \ud \tau + 2h \|v_0\| = h(2\|v_0\|+\|Av_0\|) + 2 \int_0^h \|v(\tau)-v_0\| \ud \tau.
\end{equation}
Thus Gronwall's inequality gives us
\begin{equation}
\|v(h)-v_0\|\leq h(2\|v_0\|+\|Av_0\|)e^{2h},
\end{equation}
which, together with \eqref{vLip_0}, leads to the Lipschitz continuity of $v(t)$
\begin{equation}\label{vLip}
\left\| \frac{v(t+h)-v(t)}{h} \right\|\leq 2 \|v_0\|_1 e^{2t+2h}.
\end{equation}
Then from \eqref{GLip} we concludes the Lipschitz continuity of $G(v(t))$
\begin{equation}\label{GLip_t}
\left\| \frac{G(v(t+h))-G(v(t))}{h} \right\|\leq  4 \|v_0\|_1 e^{2t+2h}.
\end{equation}
Moreover, from \eqref{tma3}
we know
   \begin{align}\label{At1}
    \frac{v(t+h)-v(t)}{h}
    =e^{-At}\frac{v(h)-v_0}{h}+ \int_0^t e^{-A(t-\tau)} \frac{G(v(\tau+h))-G(v(\tau))}{h} \ud \tau.
  \end{align}
On one hand, by the reflexibility of $L^2$ and generalized Rademacher's theorem [Evans-measure theory-th6.5??], there exists $g(t)\in L^1(0,T;L^2(\mathbb{R}))$ such that for a.e. $t\geq 0$,
\begin{equation}\label{gg}
\lim_{h\to 0}\frac{G(v(t+h))-G(v(t))}{h}= g(t),
 \end{equation}
 i.e. $g(t)=\pt_t G(v(t))$ which is the Fr\'echet derivative of $G$. Then by Lebesgue's dominated convergence theorem and \eqref{GLip_t}, we know the limit for the second term  on the right hand side of \eqref{At1}  exists.
On the other hand, from \eqref{At2},
\begin{align*}
  \frac{v(h)-v_0}{h} &= \frac{e^{-Ah}-I}{h}v_0 + \frac{1}{h} \int_0^h e^{-A(h-\tau)} [G(v(\tau))-G(v_0)+G(v_0)] \ud \tau\\
  & \to -Av_0+ G(v_0)
\end{align*}
due to continuity of $G$, so the first term on the right hand side of \eqref{At1} converges. Then by  Lebesgue's dominated convergence theorem and \eqref{vLip}, we know the limit for the second term  on the right hand side of \eqref{At1} also exists.
Therefore we know
\begin{equation}\label{v_time}
  \pt_t v(t)  = e^{-At}(-Av_0+ G(v_0)) + \int_0^t  e^{-A(t-\tau)} g(\tau) \ud \tau,
\end{equation}
which concludes $\pt_t v \in C([0,T]; L^2(\mathbb{R})).$ Plugging in the formula \eqref{gg}, \eqref{v_time} becomes
\begin{align*}
 \pt_t v (t)
 = & e^{-At}(-Av_0+ G(v_0)) + \int_0^t  e^{-A(t-\tau)} \pt_\tau G(v(\tau)) \ud \tau\\
 =& e^{-At}(-Av_0+ G(v_0)) + e^{-A(t-\tau)}G(v(\tau))|_0^t- A \int_0^t e^{-A(t-\tau)}G(v(\tau)) \ud \tau\\
 =& G(v(t)) - A e^{-At} v_0  - A \int_0^t e^{-A(t-\tau)}G(v(\tau)) \ud \tau\\
 = & -Av(t) + G(v(t))
\end{align*}
due to \eqref{mild}.
Then since $G(v(t))\in C([0,T]; L^2(\mathbb{R}))$ we have
\begin{equation}
  Av\in C([0,T]; L^2(\mathbb{R})).
\end{equation}

  Step 4. Higher order regularities.

Set $w_1=\pt_t v$ and $w_2=\pt_x v$. Then
$$\pt_t G(v)=G'(v)\pt_t v \in C([0,T]; L^2(\mathbb{R}))$$
and
$$\pt_x G(v)=(1-f'(\phi+v))\pt_x v -(f'(\phi+v)-f'(\phi))\pt_x \phi \in C([0,T]; L^2(\mathbb{R})).$$
Therefore we can repeat Step 1 and 2 for
\begin{equation}
  \pt_t w_1 + Aw_1 = G'(v)w_1
\end{equation}
and
\begin{equation}
  \pt_t w_2 +Aw_2 =(1- f'(v+\phi))w_2- (f'(\phi+v)-f'(\phi))\pt_x \phi
\end{equation}
to obtain
\begin{align*}
  &w_1, w_2 \in C((0,T];L^2(\mathbb{R}))\cap C((0,T]; H^1(\mathbb{R}))\\
  &\pt_t w_1 , \pt_t w_2  \in C((0,T]; L^2(\mathbb{R}))
\end{align*}
which concludes $v$ is a global classical solution to \eqref{Neqn} and satisfies \eqref{highreg}.

Step 5. \eqref{dissi} is directly from \eqref{Neqn} and above regularity properties. Then we have
$$\F(v(t))\leq \F(v_0).$$
\end{proof}

\section{Proof of key propositions in Section 3}\label{App_B}
\subsection{The spectrum $\sigma(L)\subset [-1, +\8)$}\label{app-s1}
\begin{lem}\label{B1}
  The operator $L$ in \eqref{L1} is densely defined, self-adjoint linear operator.
\end{lem}
\begin{proof}
  First, it is obvious $L$ is linear operator and $D(L)=H^1$ is dense in $L^2$.

  Second, denote $q(x):=\frac{x^2 -1}{x^2+1}$ which is bounded function. For any $y,\, u\in H^1$,
  $$\la Lu, y \ra= \la \ptf u + q(x)u  , y \ra = \la u, \ptf y + q(x) y \ra= \la u, L y \ra.$$
  Hence $L$ is self-adjoint.
\end{proof}

\begin{lem}\label{B2}
Assume the operator $L$ is densely defined, self-adjoint linear operator, then
  $$\kerr(\lambda I - L)= \ran(\lambda I-L)^\bot.$$
\end{lem}

\begin{proof}
  For any $y\in D(\lambda I -L)$, any $u\in \kerr(\lambda I -L)$, we have
  $$0= \la (\lambda I - L)u, y  \ra =\la u, (\lambda I - L)y  \ra. $$
\end{proof}

\begin{lem}\label{closed}
  The operator $L$ in \eqref{L1} is closed.
\end{lem}

\begin{proof}
  Assume we have $u_n \to u$ in $D(L)$ and $y_n := L u_n \to y$ in $L^2.$
  Since $D(L)$ is dense in $L^2$, we first choose any test function $\varphi\in D(L)$. Then
  \begin{align*}
   \la \varphi, y \ra \leftarrow\la \varphi, y_n \ra =\la \varphi, L u_n \ra = \la L \varphi, u_n \ra \to \la L \varphi, u\ra= \la  \varphi, L u\ra,
  \end{align*}
  and dense argument shows that $y=L u.$ Hence $L$ is closed.
\end{proof}

\begin{defn}
  Let $T$ be a closed operator on the Hilbert space $X$. A complex number $\lambda$ is in the resolvent set $\rho(T)$ if $\lambda I -T$ is bijection of $D(T)$ onto $X$ with a bounded inverse. If $\lambda\in \rho(T)$, $R_\lambda(T)=(\lambda I -T)^{-1}$ is called the resolvent of $T$ ar $\lambda.$
\end{defn}

\begin{rem}
  If $\lambda I- T$ is a bijection of $D(T)$ onto $X$, by the closed-graph theorem, its inverse is automatically bounded.
\end{rem}

\begin{lem}\label{B3}
  For linear operator $L$ in \eqref{L1} and  any $\lambda\in \mathbb{C}\backslash [-1,+\8)$, the range $\ran(\lambda I-L)$ is closed.
\end{lem}

\begin{proof}
 Notice the lower bound for potential $q(x):=\frac{x^2 -1}{x^2+1}$ is $-1$.
 For $\lambda=a+ b i$  with $b\ne 0$, we have
\begin{equation}\label{lowerbound}
     \|(\lambda I - L)u\|^2= \|b u\|^2+ \|(aI-L)u\|^2 \geq  b^2 \|u\|^2
 \end{equation}
For $\lambda <-1$, we have
\begin{equation}\label{lowerbound1}
     \|(\lambda I - L)u\|^2 =  \|(\lambda+1)u\|^2 - 2(\lambda+1)\la (L+I)u, u \ra +\|(L+I)u\|^2\\
\geq  (\lambda+1)^{2} \|u\|^2,
\end{equation}
where we used $\la (L+I)u, u \ra\geq 0$ and $\lambda+1 <0$.

Now, we show that $\ran(\lambda I-L)$ is closed. For any $y_n \in \ran(\lambda I-L)$ with $y_n=(\lambda I-L)u_n$, if $y_n\to y$, from the lower bound estimate in \eqref{lowerbound} and \eqref{lowerbound1},  $u_n \to u$. Therefore from Lemma \ref{closed} we know $y=(\lambda I -L) u$. Thus $\ran(\lambda I-L)$ is closed.
\end{proof}

\begin{lem}\label{lem2.1}
  For linear operator $L$ in \eqref{L1}, $\sigma(L)=\sigma_p(L)\cup\sigma_c(L)\subset [-1, +\8)$.
\end{lem}
\begin{proof}

  (1) For self-adjoint operator, the spectrum locates on the real line. Indeed, for any $b\neq 0$, \eqref{lowerbound}
   implies there is a lower bound for $\lambda I - L$. To obtain $\lambda \in \rho(L),$ it remains to prove $\lambda I - L$ is onto.
   If it is not onto and we assume ${\ran(\lambda I -L )} \neq L^2 $. Then by Lemma \ref{B1}, \ref{B2} and \ref{B3}, $\text{Ker}(\lambda I - L)= {\ran(\lambda I -L )}^\bot=\overline{{\ran(\lambda I -L )}}^\bot$ is not empty, which means there exists $u^*\neq 0$ such that $(\lambda I - L)u^*=0$ and it contradicts with \eqref{lowerbound}.
     \\
     (2) For self-adjoint operator, the residual spectrum is empty. Indeed, if $\lambda\in \sigma_r$, we concludes a contradiction from the fact
     $$\kerr(\lambda I - L)= \ran(\lambda I-L)^\bot\supseteq \overline{\ran(\lambda I-L)}^\bot\neq \emptyset.$$
  (3) $\sigma_p \subset [-1, +\8).$ Otherwise, if $\lambda_p=a<-1$ is a point spectrum, then \eqref{lowerbound1} leads to an contradiction.
  \\
  (4) $\sigma_c \subset [-1, +\8).$ Otherwise, if $\lambda_c=a<-1$ is a continuous spectrum, then
  $$L^2 = \overline{\ran(\lambda I -L)}= \ran(\lambda I -L),$$
  which contradicts with the definition of continuous spectrum.
\end{proof}

\subsection{The continuous spectrum $\sigma_c(L)\subset [1, +\8)$ }\label{app-s2}
It worth to mention that the proof relies only on the lower and upper bound of the potential $f'(\phi)= \frac{x^2 -1}{x^2+1}$, which  are $-1$ and $1$ separately.

\begin{lem}\label{lem-sc}
  For linear operator $L$ in \eqref{L1}, $\sigma_c(L)\subset [1, +\8)$. Besides, for the linear operator $A$ in \eqref{A2.3}, $\sigma(A)=[1,+\8).$
\end{lem}
\begin{proof}
  Recall the perturbation theorem for spectrum in \cite[Theorem XIII 14 and Corollary 2]{reed-simon} .
  \begin{center}
  \textit{
    Let $A$ be a self-adjoint operator and let $C$ be a relatively compact perturbation of $A$. Then $L:= A+C$ has the same essential spectrum with $A$.
    }
  \end{center}
  In our case, notice the upper bound for potential $q(x):=\frac{x^2 -1}{x^2+1}$ is $1$. Taking $A=\ptf + I$ and $C:= v(x)-I,$ we will first prove
  $C$ is a relatively compact perturbation of $A$, i.e. $C(A+i)^{-1}$ is compact, and then we prove $\sigma(A)=[1,+\8).$

  (1) First we prove $C(A+i)^{-1}$ is compact. Assume $u_j\in L^2$ satisfying $\|u_j\|\leq M$ for any $j$. Denote $w_j := (A+i)^{-1} u_j = ((1+i)I+\ptf )^{-1}u_j$. We want to prove for any $\eps>0$ there exist $J$ and a subsequence (still denoted as $j$) such that for any $j\geq J$ and $\ell\geq 0$, $(q(x)-1) (w_j-w_{j+\ell})$ are Cauchy sequence in $L^2$.
  \\
  (1.a) For $\eps>0$, $n:=[1/\eps]$, there exists $R_n$, such that for any $|x|>R_n$,
  \begin{equation}
    \int_{|x|>R_n} (q-1)(w_j - w_{j+\ell}) \ud x \leq \|q-1\|_{L^2(|x|>R_n)} \cdot \|w_j - w_{j+\ell}\|_{L^2(|x|>R_n)}\leq \frac{\eps}{2}.
  \end{equation}
  (1.b) For $|x|\leq R_n$, we claim $w_j= ((1+i)I+\ptf )^{-1}u_j$ is bounded in $H^1(|x|\leq R_n)$. Indeed from $((1+i)I+\ptf ) w_j= u_j$ and Fourier's transform, we know
   $$\hat{w_j}(\xi)= \frac{\hat{u}_j(\xi)}{1+i+|\xi|}.$$
   Then by Parserval's identity
   \begin{align*}
     \|w_j\|_{H^1}^2 &=\|w_j\|^2 +  \|w_j'\|^2 = \|\hat{w}_j\|^2 +  \||\xi|\hat{w}_j\|^2\\
     &= c\int \frac{1+|\xi|^2}{1+(1+|\xi|)^2} \hat{u}_j^2(\xi) \ud \xi \leq c \|u_j\|^2.
   \end{align*}
   Since $H^1(|x|\leq R_n)\hookrightarrow L^2(|x|\leq R_n)$ compactly, we obtain a subsequence (still denoted as $w_j$) of $w_j$ which strongly converges in $L^2(|x|\leq R_n)$ and $\|(q(x)-1)(w_j-w_{j+\ell})\|_{L^2(|x|\leq R_n)}\leq \frac{\eps}{2}.$
  \\
  Combining (1.a) and (1.b) gives a Cauchy sequence in $L^2$ and we conclude $C(A+i)^{-1}$ is compact.

  (2) We turn to prove $\sigma(A)=[1,+\8).$ First notice the lower bound now is $1$, so by Lemma \ref{lem2.1} $\sigma(A)\subset [1,+\8).$ It remains to prove $[0,+\8)\subset \sigma(\ptf )$ due to $A$ is a shift of $\ptf $ with constant $1$. For any $\lambda \geq 0$, we will prove $\text{Ran}(\lambda I -\ptf)\neq L^2.$ Set $f:= e^{i \xi_0 x} N(0,1)$ where $N(0,1)$ is the normal distribution. Then $f\in L^2$ and $\hat{f}(\xi)=N(\xi_0, 1)$.
  Then by Fourier's transformation, if there exists a solution to $(\lambda I -\ptf)u=f$, then $\hat{u}(\xi)= \frac{\hat{f}(\xi)}{\lambda -|\xi|}$.
  Therefore $u$ is the inverse Fourier's transform of $\frac{N(\xi_0, 1)}{\lambda-|\xi|}$ which is not integrable. Thus we have $[1, +\8)\subset \sigma(A)\subset [1,+\8)$ and $\sigma_{ess}(A)=\sigma(A)=[1,+\8).$

  Finally we conclude $\sigma_c(L)\subset \sigma_{ess}(L)=\sigma_{ess}(A)=[1,+\8).$
\end{proof}


%

\end{document}